\documentclass[10pt]{article}
\usepackage{latexsym,amsfonts,amsmath,amssymb,array,tabularx,theorem}
\usepackage{graphicx}
\usepackage[normalem]{ulem}
\usepackage[pdfpagelabels]{hyperref}
\usepackage{bbm}
\usepackage{authblk}
\usepackage{bm}
\usepackage{mathtools} 
\usepackage{dsfont} 
\usepackage{algpseudocode}

\usepackage{comment}

\DeclareFontFamily{OT1}{pzc}{}
\DeclareFontShape{OT1}{pzc}{m}{it}{<-> s * [1.10] pzcmi7t}{}
\DeclareMathAlphabet{\mathpzc}{OT1}{pzc}{m}{it}

\usepackage{float}
\newtheorem{theorem}{Theorem}[section]
\newtheorem{lemma}[theorem]{Lemma}
\newtheorem{corollary}[theorem]{Corollary}
\newtheorem{proposition}[theorem]{Proposition}
\newtheorem{remark}[theorem]{Remark}
\newtheorem{assumption}[theorem]{Assumption}

\newenvironment{proof}{\begin{trivlist}
    \item[\hskip\labelsep{\bf Proof.}]}{$\hfill\Box$\end{trivlist}}

{\theoremstyle{plain} \theorembodyfont{\rmfamily}
}

\newcommand{\bbR}{\mathbb{R}}

\newcommand{\bbN}{\mathbb{N}}

\newcommand{\bbE}{\mathbb{E}}
\newcommand{\bbF}{\mathbb{F}}
\newcommand{\bbP}{\mathbb{P}}

\newcommand{\calF}{\mathcal{F}}
\newcommand{\calO}{\mathcal{O}}

\newcommand{\calN}{\mathcal{N}}

\newcommand{\calU}{\mathcal{U}}

\newcommand{\mask}[1]{{}}
\usepackage[usenames]{color}
\definecolor{dkgreen}{rgb}{0.0, 0.5, 0.0}

\newcommand{\be}{\begin{equation}}
\newcommand{\ee}{\end{equation}}
\newcommand{\bea}{\begin{eqnarray}}
\newcommand{\eea}{\end{eqnarray}}
\newcommand{\beas}{\begin{eqnarray*}}
\newcommand{\eeas}{\end{eqnarray*}}

\DeclareMathOperator*{\argmin}{argmin}
\begin{document}

\bibliographystyle{abbrv}
\title {
Deep neural network approximation for
high-dimensional
parabolic
Hamilton--Jacobi--Bellman equations
}
\author[$\dagger$,$\ddagger$]{Philipp Grohs}

\author[$\ddagger$]{Lukas Herrmann} 

\affil[$\dagger$]{\footnotesize Faculty of Mathematics, University of Vienna, 
Oskar-Morgenstern-Platz 1,
1090 Vienna, Austria. 
\newline
\texttt{philipp.grohs@univie.ac.at}}

\affil[$\ddagger$]{\footnotesize Johann Radon Institute for Computational and Applied Mathematics, 
Austrian Academy of Sciences, 
Altenbergerstrasse 69,
 4040 Linz,
   Austria.  
   \newline
   \texttt{$\{$philipp.grohs,lukas.herrmann$\}$@ricam.oeaw.ac.at}}
\maketitle
\date{}
\begin{abstract}
The approximation of solutions to second order Hamilton--Jacobi--Bellman (HJB) equations 
by deep neural networks is investigated.
It is shown that for HJB equations that arise in the context of the optimal control of certain Markov processes
the solution can be approximated by deep neural networks 
without incurring the curse of dimension.
The dynamics is assumed to depend affinely on the controls 
and the cost depends quadratically on the controls.
The admissible controls take values in a bounded set.
\end{abstract}



\section{Introduction}

Deep neural networks (DNNs) and optimization techniques are increasingly used in the search for approximation methods for high-dimensional
problems in scientific computing.
In the case of a partial differential equation (PDE), the approximation of the solution is challenging, since grid based approaches are burdened by the so called curse of dimension
\cite{Bellman1957}.
By this we mean that in order to achieve an accuracy $\varepsilon>0$, the computational cost has asymptotically 
an exponential dependence with respect to the dimension of the spatial domain of the PDE.
As a result practical computations on dimensions larger than 4 or 6 become already very computationally intensive.

A particularly relevant example of high dimensional PDEs are Hamilton--Jacobi--Bellman (HJB) equations associated with stochastic optimal control problems which are relevant for example in engineering and 
financial modelling. 
Due to the curse of dimension, numerical methods to solve the HJB equation, 
which allows to obtain optimal feedback control using the value function and the gradient,
are usually restricted to small dimensions.
Also, the non-linearity in the HJB equation typically with respect to the gradient of the value function
constitutes a challenging obstruction.
Computational technique based on DNNs and deep learning receive increasing attention recently~\cite{HJE_2018,4267720,nakamurazimmerer2020adaptive,jiang2017using}. 
However, it is still an open question, whether the architecture of DNNs is rich enough to 
overcome the curse of dimension to solve high-dimensional (stochastic) optimal control problems.
If this is in fact the case, the mentioned computational approaches are meaningful.
In this work, we provide a proof that this is indeed possible in the cases where the dynamics depend affinely on the controls 
and the dependence of the controls in the cost function is quadratic, where admissible values of controls are bounded.

There has been vivid research in the approximation of solutions to high-dimensional PDEs with DNNs \cite{GHJv18_786,HJKN_2020,grohs2020deep}, where DNN approximation results without curse of dimension has been established. 
However, these works are either focused on linear PDEs or on non-linearities with respect to the solution but not with respect to the gradient of the solution. 
Since in the HJB equations that arise in the search for optimal feedback control for stochastic differential equations, 
the non-linearity appears naturally on the gradient of the solution, new challenges arise in the approximation by DNNs.
In order to accomodated these issues, our proof extends some ideas from~\cite{HJKN_2020} to these cases. 
Moreover the novel sampling techniques proposed in~\cite{HJK_2019} are also an important ingredient in our derivations 
to prove that there exist DNNs that can approximate the value function and the gradient related to certain parabolic HJB equations
without incurring the curse of dimension.

\subsection{Deep neural networks}

Deep neural networks are in this work considered as respresentations of mappings 
between real Euclidean spaces.
They are formed by affine mappings followed by a componentwise applied
non-linear map, which is referred to as the activation function.
We shall restrict the presentation to the case 
of DNNs with activation functions 
of so called rectified linear unit (ReLU) $\sigma_1:\bbR \to \bbR$, 
which is defined by $\sigma_1(x) := \max\{0,x\}$, $x\in \bbR$, and
so called rectified cubic unit (ReCU) $\sigma_3:\bbR \to \bbR$, 
which is defined by $\sigma_3(x) := \max\{0,x\}^3$, $x\in \bbR$.
The activation functions may change in different layers of the DNN.
Also, we restrict the presentation to fully connected DNNs of depth $L\in \bbN$ defined as follows.
Let $(N_i)_{i=0,\ldots,L}$ be a sequence of positive integers. 
Let $A^i \in \bbR^{N_{i} \times N_{i-1}}$ and $b^i\in \bbR^{N_i}$, 
$i=1,\ldots,L$.
We define the realization of the DNN $\phi^L: \bbR^{N_0} \to \bbR^{N_L}$ by 
\begin{equation}\label{eq:def_relu_nn}
\bbR^{N_0}\ni x \mapsto \phi^i(x)
:=
A^i
\sigma^{(i)}(\phi^{i-1}(x)) + b^i, 
\quad i=2,\ldots,L,
\quad 
\text{with}
\quad  
\bbR^{N_0}\ni x \mapsto \phi^1(x)
:= A^1 x + b^1, 
\end{equation}
where with slight abuse of notation we applied that  
$\sigma^{(i)} (x):= (\sigma^{(i)}(x_1),\ldots,\sigma^{(i)}(x_N)) $, $N\in\bbN$, $x\in\bbR^N$, 
with $\sigma^{(i)}\in \{\sigma_1,\sigma_3\}$.
The entries of $(A^i, b^i)_{i=1,\ldots,L}$ are referred to as the weights of the DNN $\phi^L$.
The number of 
non-zero weights is referred to as the size of the DNN $\phi^L$ and will be denoted by ${\rm size}(\phi^L)$. 
The width of the DNN $\phi^L$ is defined by ${\rm width}(\phi^L) = \max\{N_0,\ldots,N_L\}$
and $L$ is the depth of $\phi^L$.
It is sometimes argued in the literature~\cite{EDGB_2019} that one should distinguish between 
the architecture and the realization of a DNN. The reason may be that of course different architectures, 
sets of DNN weights, can result in the identical mapping that is realized in~\eqref{eq:def_relu_nn}.
Here, we do not make this distinction. The reason is that we are studying 
asymptotic upper bounds of the size that achieve a certain accuracy.
Moreover, 
in the following we omit the superscript index $L$ which is the depth of a DNN 
with the hope that this eases the notation.
In this work, we exclusively consider DNNs with ReLU and ReCU
activation function
and will mostly write ReLU DNN or ReCU DNN
in the case that the $\sigma^{(i)}=\sigma_1$
or $\sigma^{(i)}=\sigma_3$, respectively, 
for every $i=2,\ldots,L$.

Throughout this manuscript, we will construct DNNs mostly by composition and addition of already existing DNNs. 
It holds that the sum of a ReLU DNN and ReCU DNN is again in the class of DNNs
defined in~\eqref{eq:def_relu_nn}. 
We shall briefly illustrate how this may be seen in the simple example to find 
a DNN in the class of~\eqref{eq:def_relu_nn} that represents the function
$\bbR\ni x\mapsto \sigma_1(x) + \sigma_3(x)$. 
It holds that there exist ReLU DNNs and ReCU DNNs of constant size
such that the univariate identity function may be represented exactly.
For ReLU DNNs, it holds that $x = \sigma_1(x) - \sigma_1(-x)$ for
every $x\in\bbR$. 
In case of the ReCU DNNs, we exploit the fact that 
$x^3 = \sigma_3(x) - \sigma_3(-x)$ and the identity 
$24x = (x+2)^3 - 2x^3 + (x-2)^3$, $x\in\bbR$, 
which implies that for every $x\in\bbR$,
\begin{equation*}
    x = \frac{1}{24}
    \left( 
    \sigma_3(x+2 ) - \sigma_3(-x-2)
    -2(\sigma_3(x) - \sigma_3(-x))
    +\sigma_3(x-2) - \sigma_3(-x+2)
    \right).
\end{equation*}
The theoretical formalism of \emph{parallelization} of DNNs 
allows to form a DNN that executes a finite number, here two, DNNs in parallel, cf.~\cite[Lemma~II.5]{EDGB_2019}. 
This may be written in column vector notation. 
Thus, the following schematic mappings are all DNNs according to the definition in~\eqref{eq:def_relu_nn},
\begin{equation*}
    \bbR\ni x
    \mapsto 
    \begin{pmatrix}
    x \\ x
    \end{pmatrix}
    \mapsto
    \begin{pmatrix}
    \sigma_1(x)\\ x
    \end{pmatrix}
    \mapsto 
    \begin{pmatrix}
    \sigma_1(x)\\ \sigma_3(x)
    \end{pmatrix}
    \mapsto 
    \sigma_1(x) + \sigma_3(x)
    .
\end{equation*}
The resulting DNN is a composition of 
four DNNs, the second one is a ReLU DNN, and the third one a ReCU DNN.

The asserted upper bounds in later parts of the manuscript on the size of certain DNNs then result
for example by~\cite[Lemmas~II.5 and~II.6]{EDGB_2019}.
Similar statements hold for ReCU DNNs.

\subsection{Parabolic Hamilton--Jacobi--Bellman equations and description of the main result} 
We can now describe our main result.
Let $(\Omega, \bbF,\calF,\bbP)$
be a filtered probability space
and denote the expectation with respect to
$\bbP$ by $\bbE(\cdot)$.
Let $d,\bar{d}\in\bbN$.
We consider the $\bbR^d$-valued Markov process $x(t)$
defined by the stochastic differential equation
\begin{equation}\label{eq:SDE}
 {\rm d} x(t) =
 f(t,x(t),u(t))
 {\rm d}t
 +
 {\rm d} W_t,
 \quad t \in [0,t_f],
 \quad x(0) = 0,
\end{equation}
where $W_t$ is a standard $d$-dimensional Brownian motion, 
$u(t)$ is a progressively measurable process that takes values in the closed set $U\subset \bbR^{\bar{d}}$,
and $f$ is globally Lipschitz continuous.
The process $u(t)$ is referred to as control. 
Let us denote the progessively measurable processes that take values in $U$ by $\calU$.
The resulting optimal control problem is to determine, for each $x\in \bbR^d$ 
a control $u^{\ast}(\cdot , x)\in \calU$ with 
$$
    u^\ast(t,x) \in \argmin_{{\{u(t)\}_{t\in [0,t_f]}\in \calU}}\bbE\left(
    \int_{0}^{t_f}
 L(s,x(s),u(s))
 {\rm d}s
 +\Psi(x(t_f))
 \Big \vert x(0) =x\right)
$$
for suitable cost functions $L$ and $\Psi$ and terminal time $t_f$.

A popular method to determine $u^\ast$ is via the associated Hamilton--Jacobi--Bellman equation. To this end, let 
\begin{equation*}
 J(t,x;u) :=
 \bbE\left( 
 \int_{t}^{t_f}
 L(s,x(s),u(s))
 {\rm d}s
 +\Psi(x(t_f))
 \Big \vert x(t) =x
 \right)
\end{equation*}
and define the \emph{value function}
\begin{equation}\label{eq:value_fctn}
 V(t,x) 
 :=
 \inf_{\{u(t)\}_{t\in [0,t_f]}\in \calU}
 J(t,x;u)
 .
\end{equation}
The value function satisfies the following Hamilton--Jacobi--Bellman equation
with terminal condition
\begin{equation}\label{eq:HJB_PDE}
 \partial_t V(t,x)
 + \frac{1}{2}\Delta V(t,x) + H(t,x,\nabla_x V(t,x))=0
  \text{ on } [0,t_f)\times \bbR^d, 
  \quad V(t_f,x) = \Psi(x) \text{ on } \bbR^d
  ,
\end{equation}
where the \emph{Hamiltonian} $H$ is given by
\begin{equation}\label{eq:Hamiltonian}
 H(t,x,p) 
 :=
 \inf_{v\in U}
 [p^\top f(t,x,v) + L(t,x,v) ]
 .
\end{equation} 
Given a solution $V(t,x)$ of (\ref{eq:HJB_PDE}) an \emph{optimal policy function} can be constructed by choosing an element
\begin{equation}\label{eq:alpha}
    \alpha(t,x) \in \argmin_{v\in U} [\nabla_x V(t,x)^\top f(t,x,v) + L(t,x,v) ]
\end{equation}
and computing
\begin{equation*}
 {\rm d} x^\ast(t) =
 f(t,x^\ast(t),\alpha(t,x^\ast(t)))
 {\rm d}t
 +
 {\rm d} W(t),
 \quad t \in [0,t_f],\quad x^\ast(0) = x.
\end{equation*}
An optimal control $u^\ast$ is then given by
\begin{equation}
    u^\ast(t,x)=\alpha(t,x^\ast(t)).   
\end{equation}
The key to the success of this approach is to have efficient representations of the value function $V$, the gradient $\nabla_x V$, as well as the function $\alpha$ in (\ref{eq:alpha}). The main difficulty lies in the typically high dimension $d$ of the configuration space $\bbR^d$ and the fact that most classical numerical representations, such as finite elements or spectral elements, suffer from the curse of dimension. In the present paper we will show that this impediment is not true for a DNN representation. In particular, we will show that, given suitable assumptions, for each $t\in [0,t_f]$ the value function $V(t,\cdot)$, the gradient $\nabla_x V(t,\cdot)$, as well as the function $\alpha(t,\cdot)$ in (\ref{eq:alpha}) can be represented by DNNs without incurring the curse of dimension, provided that similar approximation results hold for the functions $f,L,\Psi$. 
For more details on stochastic optimal control problems, 
the reader is referred to~\cite{FlemingSoner}.
In the following, we will focus on particular dynamics in~\eqref{eq:SDE} as laid out in the following assumption.
\begin{assumption}\label{ass:lineardynamics}
Assume that the dynamics $f$ depends affinely on the controls and that in the cost function $L$
the dependence with respect to the controls is quadratic.
Specifically, 
let
\begin{equation}\label{eq:affine_controls}
 f(t,x,v) =    
    f_1(t,x) + f_2(t,x)v
    \quad 
    \forall (t,x,v) \in [0,t_f]\times \bbR^d \times U
    ,
\end{equation}
where $f_1:[0,t_f]\times \bbR^d \to \bbR^d$ 
and $f_2:[0,t_f]\times \bbR^d \to \bbR^{d\times \bar{d}}$
are continuous. 
Moreover suppose that 
\begin{equation}\label{eq:contols_in_box}
    U= 
    [a_1,b_1] \times \cdots \times [a_{\bar{d}}, b_{\bar{d}}]
\end{equation}
for real numbers $a_i < b_i$, $i=1,\ldots,\bar{d}$.
The cost function is assumed to have the form for $\gamma>0$
\begin{equation}\label{eq:quadratic_cost}
    L(t,x,v)
    =
    \bar{L}(t,x) + \gamma\|v\|_2^2
    \quad 
    \forall (t,x,v) \in [0,t_f]\times \bbR^d \times U,
\end{equation}
and $L $ is continuous.
Let $\Psi\in C^2(\bbR^d)$.
There exist constants $C,\kappa_1,\kappa_2>0$ such that
for every $d,\bar{d}\in\bbN$,
\begin{enumerate}
    \item $
    \sup_{s\in [0,t_f]}\sup_{x\in\bbR^d}
    \|f_1(s,x)\|_1 
    + 
    \|f_2(s,x)\|_1 
     \leq 
     C
    ,$
    \item 
    $\sup_{s\in [0,t_f]}\|f_1(s,x)
    -f_1(s,x')\|_2
    \leq C\|x-x'\|_2$,
    \item 
    $\sup_{s\in [0,t_f]}\|f_2(s,x)
    -f_2(s,x')\|_2
    \leq C d^{\kappa_1} \bar{d}^{\kappa_2}\|x-x'\|_2$,
    \item 
    $\sup_{x\in\bbR^d}|\Psi(x)|
    +
    \sup_{s\in [0,t_f]}\sup_{x\in\bbR^d}
    (\bar{L}(s,x)
    +
    \|\nabla_x \bar{L}(s,x)\|_2)
    \leq 
    C d^{\kappa_1} \bar{d}^{\kappa_2} $,
    \item $\sup_{x\in\bbR^d} \|\nabla_x \Psi\|_1\leq 
    C d^{\kappa_1} \bar{d}^{\kappa_2}
    $
    .
\end{enumerate}
\end{assumption}
\begin{remark}\label{rem:lineardynamics} 
Given the specific dynamics in Assumption \ref{ass:lineardynamics}, the Hamiltonian possesses a simple explicit expression than can be easily evaluated by a small DNN, see Lemma \ref{lem:formula_H}. This is the main reason for us to consider this restricted setting. We expect most of our results to hold true in a more general setting, specifically  under certain growth condition of $f$ in $x$ and $v$ together with the assumption that the function $p^\top f(t,x,v) + L(t,x,v)$ is uniformly strongly convex in $v$. We leave the detailed proof to future work.
\end{remark}

We are now ready to state the main theorem of this paper stating that, given our assumptions, the value function $V$, the gradient $\nabla_x V$, as well as the optimal policy function $\alpha$ can be approximated by DNNs without curse of dimension.
\begin{theorem}
\label{thm:main_resutl}
Let Assumption \ref{ass:lineardynamics} be satisfied.
Let $q\in [1,\infty)$.
Suppose that for every $\delta_\Psi \in (0,1)$ there exist
ReCU DNNs $\phi_{\Psi,\delta_\Psi}$ 
with corresponding Lipschitz constant $C_{\phi_\Psi,x}$
that
satisfy
that for every $x\in\bbR^d$
\begin{equation*}
    |\Psi(x) - \phi_{\Psi,\delta_\Psi}(x)|
    \leq \delta_\Psi(1 + \|x\|_2^q).
\end{equation*}
Suppose for every $\delta_1,\delta_2,\delta_{\bar{L}}\in (0,1)$, there exist DNNs $\phi_{f_1,\delta_1}$,
 $\phi_{f_2,\delta_2}$, $\phi_{\bar{L},\delta_{\bar{L}}}$ with corresponding
 Lipschitz constants 
 $C_{1,x}$, $C_{2,x}$, $C_{\bar{L},x}$
 such that
 for every $t\in [0,t_f]$
 and every $x\in \bbR^d$
 \begin{equation*}
     \|f_{1}(t,x) - \phi_{f_1,\delta_1}(t,x)\|_2,
     \leq \delta_1(1 + \|x\|_2^q),
 \end{equation*}
 \begin{equation*}
     \|f_{2}(t,x) - \phi_{f_2,\delta_2}(t,x)\|_2 \leq \delta_2(1 + \|x\|_2^q),
 \end{equation*}
 and 
 \begin{equation*}
     |\bar{L}(t,x) - \phi_{\bar{L},\delta_{\bar{L}}}(t,x)| \leq \delta_{\bar{L}}(1 + \|x\|_2^q).
 \end{equation*}
 Suppose there exists a constant $\kappa_0>0$ 
 that does not depend on $d,\bar{d}$ such that 
 \begin{equation*}
 \sup_{\delta_{\Psi},\delta_1,\delta_2,\delta_{\bar{L}}\in (0,1)}
     \sup_{t\in[0,t_f]}
     \sup_{x\in\bbR^d}
     \{
     \|\phi_{f_1,\delta_1}(t,x)\|_\infty
     +
     \|\phi_{f_2,\delta_2}(t,x)\|_\infty
     +\|\phi_{\bar{L},\delta_{\bar{L}}}(t,x)\|_\infty
     \}
     = \calO(d^{\kappa_0} \bar{d}^{\kappa_0}).
 \end{equation*}
  Furthermore, we suppose that the Lipschitz constants satisfy
 \begin{equation*}
 \sup_{\delta_{\Psi},\delta_1,\delta_2,\delta_{\bar{L}}\in (0,1)}
 \max\{C_{\phi_\Psi,x},C_{1,x}, C_{2,x}, C_{\bar{L},x}\} = \calO(d^{\kappa_0} \bar{d}^{\kappa_0}).
 \end{equation*}
 Moreover, we assume that 
 \begin{equation*}
 \max\{  {\rm  size}(\phi_{\phi_\Psi,\delta_\Psi}),
 {\rm  size}(\phi_{f_1,\delta_1}),
 {\rm  size}(\phi_{f_2,\delta_2}), 
 {\rm  size}(\phi_{\bar{L},\delta_{\bar{L}}})
 \}
 = \calO( d^{\kappa_0}\bar{d}^{\kappa_0} \min\{\delta_\Psi,\delta_1,\delta_2,\delta_{\bar{L}}\}^{-\kappa_0}).
 \end{equation*}
 %
Then, for every $\varepsilon>0$ and every bounded domain $Q\subset\bbR^d$ there exist DNNs $\phi_\varepsilon$, $\phi_{\varepsilon,\nabla}$, $\phi_{\varepsilon,\alpha}$
such that
\begin{equation*}
 \| V(0,\cdot) - \phi_\varepsilon\|_{L^2(Q)} 
 +
 \left\|\|\nabla_x V(0,\cdot) - \phi_{\varepsilon,\nabla}\|_2\right\|_{L^2(Q)}
 +
 \left\|\|\alpha(0,\cdot) - \phi_{\varepsilon,\alpha}\|_2\right\|_{L^2(Q)}
 \leq \varepsilon
\end{equation*}
and 
\begin{equation*}
\max\left\{{\rm size}(\phi_\varepsilon) ,
{\rm size}(\phi_{\varepsilon,\nabla}) ,
{\rm size}(\phi_{\varepsilon,\alpha})\right\}=\calO\left(\sup_{x\in Q}\{\|x\|_2^\kappa\} |Q|^\kappa d^\kappa \bar{d}^\kappa \varepsilon^{-\kappa}\right)
\end{equation*}
for some $\kappa>0$ that does not depend on $d,\bar{d}$.
\end{theorem}

Theorem \ref{thm:main_resutl} follows directly from  
Theorem~\ref{thm:DNN_approx_value_fct} and Corollary~\ref{cor:DNN_approx_control} in Section~\ref{sec:DNN_approx_Value_fct}.
\begin{remark}
This estimate in Theorem~\ref{thm:main_resutl} is pointwise with respect to the temporal argument. 
It can be turned into a space-time estimate with a DNN that takes the temporal and spatial variable as input.
Similar techniques as in~\cite{grohs2019spacetime}
may be applied. 
However, adaptations would still be needed which is why we leave this extension to future work.
\end{remark}
\subsection{Architecture of the Proof}

The approach of the proof of the main result Theorem~\ref{thm:main_resutl} will be carried out in various lemmas, propositions, and finally be concluded 
in Section~\ref{sec:DNN_approx_Value_fct}.
The HJB equation in~\eqref{eq:HJB_PDE}
is governed by the Hamiltonian $H$, 
which Lipschitz constant is generally linearly growing with respect to the third argument. 
The first step will be to utilize that in our setup under Assumption~\ref{ass:lineardynamics}, 
the gradient of the value function is bounded. 
This allows us to truncate the Hamiltonian in the third variable, 
and thus obtain a truncated Hamiltonian, which is globally Lipschitz continuous in the second
and third argument 
and induces the same value function. 
The second step comprises the construction of a DNN approximation of the Hamiltonian based on DNNs that approximate the dynamics 
from Assumption~\ref{ass:lineardynamics}.
The established global Lipschitz continuity of the Hamiltonian enables us to use the sampling technique \emph{multilevel Picard} (MLP) method
that was introduced in~\cite{HJK_2019,HJKN_2020}.
We will control the error incurred by this 
perturbation of the Hamiltonian by a DNN approximation
to the value function and 
then show existence of DNN weights of the resulting DNN that approximate the value function and the gradient
without the curse of dimension.
The MLP method will as the last step be a vehicle in the proof to show the existence of the DNN weights.

\subsection{Outline and notation}

In Section~\ref{sec:trunc_H}, 
we show that a truncated Hamiltonian may be considered and yield in our cases 
the same value function. The truncated Hamiltonian is globally Lipschitz continuous.
In Section~\ref{sec:perturb_H}, we prove that the value function 
that results by perturbation of the Hamiltonian and the terminal condition
approximates the original value function.
In Section~\ref{sec:DNN_approx_H}, we show that the Hamiltonian may be approximated by DNNs
under the assumption that the dynamics of the Markov process and cost function 
can be approximated by DNNs.
In Section~\ref{sec:DNN_approx_Value_fct}, we prove the main result that there exist DNNs 
that approximate the value function and the gradient without incurring the curse of dimension.

We shall denote the set of real-valued functions on $[0,t_f]\times\bbR^d$
which are once continuously differentiable in the first argument 
and twice continuously differentiable in the second argument 
at every $(t,x)\in [0,t_f]\times \bbR^d$ by 
$C^{1,2}([0,t_f]\times \bbR^d)$.
The twice continuously differentiable functions on $\bbR^d$ are denoted by $C^2(\bbR^d)$.
For any bounded Lipschitz domain $\mathcal{D}\subset \bbR^n$, $n\in\bbN$,
and $\gamma\in [0,\infty)$, 
we denote the H\"older spaces by $C^\gamma(\overline{\mathcal{D}})$.
For every $q\in [1,\infty]$, we denote 
the $q$-norm
by $\|\cdot\|_q$, i.e.,
for every $A\in\bbR^{n\times m}$
$\|A\|_q := (\sum_{i=1}^n\sum_{j=1}^m
|A_{ij}|^q)^{1/q}$, $q\in [1,\infty)$,
and $\|A\|_{\infty} = \max_{i=1,\ldots,n}\max_{j=1,\ldots,m}\{|A_{ij}|\}$,
$m,n\in\bbN$.
For any bounded domain $Q\subset\bbR^n$, $n\in \bbN$,
the Lebesgue measure of $Q$ is denoted by
$|Q|$.
The function space of square integrable functions on $Q$
is denoted by $L^2(Q)$.
The normal distribution with mean $\mu$ and variance $\sigma^2>0$
is denoted by $\calN(\mu, \sigma^2)$. 
Finally the gradient with respect to $x$ or $p$
is denoted by $\nabla_x$ and $\nabla_p$, 
respectively.

\section{Truncating the Hamiltonian}
\label{sec:trunc_H}

A major obstruction in the analysis of solutions to HJB equations is the Hamiltonian, 
which is in general not globally Lipschitz continuous. 
In our setup under Assumption~\ref{ass:lineardynamics}, we may circumvent this issue by a suitable truncation 
of the Hamiltonian.

The value function in~\eqref{eq:value_fctn} is the unique solution to~\eqref{eq:HJB_PDE} under Assumption~\ref{ass:lineardynamics}. 
Specifically, by~\cite[Theorem~3.1]{HJB_classical_solutions}
existence and uniqueness of the value
function follows in $C^{1,2}([0,t_f]\times \bbR^d)$. 
Moreover, 
the gradient of the value function is uniformly bounded over the domain $[0,t_f]\times \bbR^d$.
Specifically, under Assumption~\ref{ass:lineardynamics}, Corollary~\ref{cor:bound_grad_V} implies that
there exist $C,\bar{\kappa},\kappa_1,\kappa_2$ such that 
for every $d,\bar{d}$
\begin{equation}\label{eq:bound_grad_V}
 \sup_{t\in [0,t_f]} \sup_{x\in \bbR^d} \| \nabla_x V(t,x)\|_2
 \leq 
 C e^{\bar{\kappa} t_f} 
 d^{\kappa_1} \bar{d}^{\kappa_2}
\end{equation}
We seek to utilize that the gradient of the solution $V$ is bounded and shall modify $H$ in the third variable
where it does not influence the solution $V$ as a vehicle in the theoretical arguments in the following.
For any $R>0$, let us define the function 
\begin{equation*}
 \chi_R(y) = 
 \begin{cases}
 \min\{y,R\} & \text{ if } y\geq 0,\\
 \max\{y,-R\} & \text{ if } y <0 .
 \end{cases}
 \end{equation*}
For simplicity, we use the same notation if $\chi_R$ is applied to vectors coordinatewise
and write $\chi_R(p) = (\chi_R(p_1),\ldots, \chi_R(p_d))^\top  $
for any $p\in\bbR^d$, $R>0$.
Let us define a globally Lipschitz continuous version $H_R$ of $H$ 
by
\begin{equation*}
 (t,x,p)
 \mapsto 
 H_R(t,x,p)
 :=
 \inf_{v\in U}
 \{
 \chi_R(p)^\top f(t,x,v) + L(t,x,v)\}
 \quad \forall (t,x,p) \in [0,t_f]\times\bbR^d\times\bbR^d 
 .
\end{equation*}
By construction, it holds that $H(t,x,p) = H_R(t,x,p)$ 
for every $(t,x,p) \in [0,t_f]\times\bbR^d \times [-R,R]^d$.
Since $V\in C^{1,2}([0,t_f]\times \bbR^d)$ 
is a classical solution of~\eqref{eq:HJB_PDE} with bounded first order spatial derivatives, for the choice 
\begin{equation}\label{eq:def_R}
R:= \sup_{t\in [0,t_f]}\sup_{x\in\bbR^d} \|\nabla_x V(t,x)\|_\infty <\infty,
\end{equation}
$V$ is also a classical solution to
\begin{equation*}
 \partial_t V(t,x)
 + \frac{1}{2}\Delta V(t,x) + H_R(t,x,\nabla_x V(t,x))=0
  \text{ on } [0,t_f)\times \bbR^d, 
  \quad V(t_f,x) = \Psi(x) \text{ on } \bbR^d
  .
\end{equation*}
Moreover, by~\eqref{eq:bound_grad_V} it holds that
\begin{equation}\label{eq:bound_R}
    R \leq 
    C e^{\bar{\kappa} t_f} 
 d^{\kappa_1} \bar{d}^{\kappa_2}
 ,
\end{equation}
where the constants $C,\bar{\kappa},\kappa_1,\kappa_2 >0$
do not depend on $d,\bar{d}$.

As we shall see ahead in Corollary~\ref{cor:Lipschitz_H_R}, the truncated Hamiltonian $H_R$ is globally Lipschitz continuous with respect to the second 
and third variable. We shall below (see ahead Lemma~\ref{lem:Lipschitz_H})
also quantify the Lipschitz constants and make the dependence on the dimension $d$
 and the dimension $\bar{d}$ of the set of admissible controls explicit.
By \cite[Lemma~4.2 (ii)]{HJK_2019} the solution 
to~\eqref{eq:HJB_PDE} and its gradient may be represented as
\begin{equation}\label{eq:stoch_repr_HJB}
\begin{aligned}
(V(s,x), \nabla_x V(s,x))^\top
&=
\bbE\left(\Psi(x + W_{t_f-s}) (1,\tfrac{W_{t_f-s}}{t_f-s})^\top \right)
\\
&\quad + 
\bbE\left( 
\int_{s}^{t_f}
H_R(t,x + W_{t-s}, \nabla_x V(t,x + W_{t-s}))(1,\tfrac{W_{t-s}}{t-s})^\top
{\rm d}t  
\right)
.
\end{aligned}
\end{equation}
Equation~\eqref{eq:stoch_repr_HJB} also holds when $H_R$
is replaced by $H$, $d,\bar{d}\in\bbN$.

We state in the following lemma 
an expression for the Hamiltonian, which holds under Assumption~\ref{ass:lineardynamics}.

\begin{lemma}\label{lem:formula_H}
Let Assumption \ref{ass:lineardynamics} be satisfied.
Then for every $(t,x,p)\in [0,t_f]\times \bbR^d \times \bbR^d$ it holds that
\begin{equation*}
    H(t,x,p)
    =
    p^\top f(t,x,\bar{u}(t,x,p)) + L(t,x,\bar{u}(t,x,p))
\end{equation*}
for 
\begin{equation*}
    \bar{u}(t,x,p)_i
    =
    \min\{ \max\{- (f_2(t,x)^\top p )_i /(2\gamma)  , a_i\}, b_i\} 
    \quad 
    i=1,\ldots, \bar{d}
    .
\end{equation*}
\end{lemma}

\begin{proof}
 The specific dependence on the controls allows to separate the dependencies within the coordinates of the controls.
 Specifically, it suffices to minimize
 \begin{equation*}
    \argmin_{v_i \in [a_i,b_i]}\{ f_2(t,x)^\top p v_i + \gamma v_i^2\}
     \quad i = 1 ,\ldots, \bar{d}
     .
 \end{equation*}
 This is solved by $\bar{v}_i = -(f_2(t,x)^\top p )_i /(2\gamma)$ if 
 $-(f_2(t,x)^\top p )_i /(2\gamma) \in [a_i,b_i]$, $i=1,\ldots,\bar{d}$. 
 In the other cases this value needs to be projected to the interval $[a_i,b_i]$,
 $i=1,\ldots,\bar{d}$. The assertion is thus proved.
\end{proof}

\begin{lemma}\label{lem:Lipschitz_H}
Let Assumption~\ref{ass:lineardynamics}(i) be satisfied.
There exists a constant $\widetilde{C}>0$ such that
for every $x,x'\in \bbR^d$
and every $p,p'\in \bbR^d$, $d,\bar{d}\in\bbN$,
\begin{equation*}
    |H(t,x,p) - H(t,x',p')| 
    \leq 
    \bar{C} \|x-x'\|_1
    +
    \widetilde{C}\| p- p'\|_\infty,
\end{equation*}
where 
\begin{equation*}
\begin{aligned}
    \bar{C}
    &=
    \max\{\|p\|_2,\|p'\|_2\}
    \\
    &\times
    \max_{i=1,\ldots,d}
    \left(\|\partial_{x_j} f_1(t,x)\|_2 
    + \|\partial_{x_j} f_2(t,x)\|_2\left(\max\{\|a\|_2, \|b\|_2\}
    + \frac{\|f_2(t,x)\|_2\|p\|_2}{2\gamma}\right) 
    \right)
     \\
     &\quad 
    +
    \|\nabla_x \bar{L}(t,x)\|_\infty
    + 
    \max_{i=1,\ldots,d}\frac{\|\partial_{x_j}f_2(t,x)\|_2 \max\{\|p\|_2,\|p'\|_2\}}{2\gamma} \max\{\|a\|_2, \|b\|_2\}
    .
    \end{aligned}
\end{equation*}
\end{lemma}
\begin{proof}
 The gradient of $H$ with respect to $p$ is given by
 \begin{equation}\label{eq:grad_p_H}
     \nabla_p H(t,x,p)
     =
     f_1(t,x)
     +
     f_2(t,x) \bar{u}(t,x,p)
     + p^\top f_2(t,x) D_p \bar{u}(t,x,p)
     +\gamma \bar{u}(t,x,p) D_p\bar{u}(t,x,p)
     ,
 \end{equation}
where $\bar{u}(t,x,p)$
 was specified in Lemma~\ref{lem:formula_H}
 and $D_p$ denotes the Jacobian with respect to $p$.
Moreover, the weak derivative of $\bar{u}$
satisfies
\begin{equation*}
    \partial_{p_j} \bar{u}(t,x,p)_i
    =
    \begin{cases}
    - f_2(t,x)_{ij}/(2\gamma) & \text{if } (-f_2(t,x)^\top p)_i/(2\gamma) \in [a_i,b_i],
    \\
    0 &\text{else}.
    \end{cases}
\end{equation*}
We shall bound $\|\nabla_p H(t,x,p)\|_1$. 
In particular, by the assumption on $f_2$, 
the second term in~\eqref{eq:grad_p_H} may be estimated
$\|f_2(t,x)\bar{u}(t,x,p)\|_1\leq C \max\{\|a\|_\infty,\|b\|_\infty\}$.
To estimate the third term, we use the properties of the support of $D_p \bar{u}$
and obtain 
\begin{equation*}
    \| p^\top f_2(t,x) D_p \bar{u}(t,x,p) \|_1
    \leq 
    \max\{\|a\|_\infty,\|b\|_\infty\}
    \sum_{i=1}^{d}\sum_{j=1}^{\bar{d}}
    \frac{1}{2\gamma}|f_2(t,x)_{ij}|
    \leq 
    \frac{\max\{\|a\|_\infty,\|b\|_\infty\}}{2\gamma}
    C.
\end{equation*}
Similarly, we obtain that 
$\|\bar{u}(t,x,p) D_p\bar{u}(t,x,p)\|_1\leq C\max\{\|a\|_\infty,\|b\|_\infty\}$.
This establishes the Lipschitz bound with respect to $p$ and the $\infty$-norm.

In the second step, we show Lipschitz continuity with respect to the second variable of $H$. 
Note that 
\begin{equation*}
    \partial_{x_j}
    \bar{u}(t,x,p)_i
    =
    \begin{cases}
    -(\partial_{x_j}(f_2)^\top p)_i/(2\gamma)
    &
    \text{if }
    -((f_2)^\top p)_i/(2\gamma) \in [a_i,b_i],\\
    0 & \text{else.}
    \end{cases}
\end{equation*}
It holds that
\begin{equation*}
    \nabla_x H(t,x,p)
    =
    (D_x f_1(t,x)
    + D_x(f_2(t,x) \bar{u}(t,x,p))^\top p
    + \nabla_x \bar{L}(t,x)
    +
    2\gamma 
    D_x(\bar{u}(t,x,p))^\top \bar{u}(t,x,p),
\end{equation*}
which implies 
\begin{equation*}
\begin{aligned}
    &\| \nabla_x H(t,x,p)\|_\infty
    \\
    &\leq 
    \|p\|_2
    \max_{i=1,\ldots,d}
    \left\{\|\partial_{x_j} f_1(t,x)\|_2 
    + \|\partial_{x_j} f_2(t,x)\|_2\left(\max\{\|a\|_2, \|b\|_2\}
    + \frac{\|f_2(t,x)\|_2\|p\|_2}{2\gamma}\right) 
    \right\}
    \\
    &\quad +
    \|\nabla_x \bar{L}(t,x)\|_\infty
    + 
    \max_{i=1,\ldots,d}\frac{\|\partial_{x_j}f_2(t,x)\|_2\|p\|_2}{2\gamma} \max\{\|a\|_2, \|b\|_2\}
    ,
\end{aligned}
\end{equation*}
where the Jacobian with respect to $x$ is denoted by $D_x$. 
The assertion follows by the fundamental theorem of calculus.
\end{proof}

Let us state the following consequence for $H_R$
as a corollary to the previous lemma.

\begin{corollary}\label{cor:Lipschitz_H_R}
Let Assumption~\ref{ass:lineardynamics}(i)--(iv) be satisfied.
Then, there exist $C',\kappa_3,\kappa_4>0$
such that for every $t\in [0,t_f]$, 
and every $d,\bar{d}\in\bbN$, $x,x',p,p'\in\bbR^d$
\begin{equation*}
    |H_R(t,x,p) - H_R(t,x',p')| 
    \leq 
    C' R d^{\kappa_3} \bar{d}^{\kappa_4} \|x-x'\|_1
    +
    C'\| p- p'\|_\infty
    .
\end{equation*}
\end{corollary}

\section{Continuous dependence for perturbed Hamiltonian}
\label{sec:perturb_H}

In this section, we analyze the impact of the solution
$V$ and the gradient $\nabla_x V$, when 
the Hamiltonian $H$ and the terminal condition $\Psi$ are perturbed. 
This shall in a later section be applied, 
when the Hamiltonian is approximated by a DNN.
However, we shall formulate the resulting statement
in this section in a general way.

Let us recall concentration bounds of Gaussian vectors in the 
$\infty$-norm.
\begin{lemma}\label{lem:concentration_gauss_maxnorm}
Let $Z=(z_1,\ldots,z_n)^\top$ be a Gaussian vector, i.e., 
$z_i$ is normally distributed with mean zero and variance equal to $\sigma^2$
for some $\sigma>0$.
Then, 
for every $p\in [1,\infty)$
there exists a constant $C_p$ that only depends on $p$
such that for every $\alpha>0$
 \begin{equation*}
\begin{aligned}
\bbE\left( \left |\| Z\|_\infty - \| Z\|_\infty \mathbbm{1}_{\left\{\|Z\|_\infty \leq \sigma \sqrt{ 2\log(2n )  } + \alpha\right\}}\right |^p \right)^{1/p}
&\leq 
\sigma C_p  \sqrt{\log(2n)}
e^{-\alpha^2/(4p\sigma^2)}
.
\end{aligned}
\end{equation*}
.
\end{lemma}

\begin{proof}
 We recall some basic facts 
on the $\infty$-norm of Gaussian vectors.
It holds that
\begin{equation}\label{eq:mean_Zmaxnorm}
\bbE\left( 
\| Z\|_\infty
\right)
\leq 
 \sigma \sqrt{2\log(2n) }
\end{equation}
for any $n$-dimensional random vector $Z=(z_1,\ldots, z_n)$ such that $z_i\sim \calN (0,\sigma^2)$
for some $\sigma > 0$, $i=1,\ldots,n$.
Specifically, for every $\kappa >0 $ by the Jensen inequality
\begin{equation*}
\exp(\kappa \bbE (\|Z\|_\infty) )
\leq 
\bbE( \exp( \kappa \|Z\|_{\infty}) )
\leq 
\bbE\left(
\max_{i=1,\ldots, n}
\exp(\kappa  |z_i|) )
\right)  
\leq
n 
\bbE(\exp(\kappa  |z_1|) )
\leq
2n e^{\kappa^2 \sigma^2/2} 
.
\end{equation*}
The choice $ \kappa = \sqrt{2 \log(2n)}/\sigma $ and applying the logarithm
implies 
\begin{equation*}
\bbE(\| Z \|_\infty)
\leq 
\sigma \sqrt{ 2\log(2n )  }
.
\end{equation*}
Moreover, generally there holds
\begin{equation}\label{eq:lp_estimate_max_norm}
\bbE(\| Z \|_\infty^p)
\leq 
C_p 
\sigma^p  \log(2n )^{p/2}  ,
\end{equation}
where the constant $C_p$ only depends on $p$, see for example~\cite[Lemma~A.1]{Chatterjee_2014}.
The deviation from the mean of $\|Z\|_\infty$ satisfies (see for example~\cite[Equation~(A.7)]{Chatterjee_2014})
that for every $\alpha \geq 0$
\begin{equation}\label{eq:deviation_max_norm}
\bbP\left(
\|Z\|_\infty\geq \sigma \sqrt{ 2\log(2n )  } + \alpha
\right)
\leq
\bbP\left(
\|Z\|_\infty\geq \bbE(\|Z\|_\infty ) + \alpha
\right)
\leq 
e^{-\alpha^2/(2\sigma^2)}
,
\end{equation}
where we used that $\|Z\|_\infty = \max\{z_1,-z_1,\ldots, z_n, -z_n \}$
and applied \eqref{eq:mean_Zmaxnorm}.

We truncate the range of the weight $\|Z\|_\infty$ by $ \sigma \sqrt{ 2\log(2n )  }+ \alpha$
for some $\alpha >0$
and obtain with the Cauchy--Schwarz inequality, \eqref{eq:lp_estimate_max_norm}, and~\eqref{eq:deviation_max_norm}
for every $p\in [1,\infty)$,
\begin{equation*}
\begin{aligned}
&\bbE\left( \left |\| Z\|_\infty - \| Z\|_\infty \mathbbm{1}_{\left\{\|Z\|_\infty \leq \sigma \sqrt{ 2\log(2n )  } + \alpha\right\}}\right |^p \right)^{1/p}
\\
&\qquad\leq 
\left(  \bbE( \|Z\|_\infty^{2p} )  \right)^{1/{2p}}
\left( 
\bbE\left(\mathbbm{1}_{\left\{\|Z\|_\infty > \sigma \sqrt{ 2\log(2n )  } + \alpha\right\}} \right)
\right)^{1/(2p)}
\leq \sigma C_p  \sqrt{\log(2n)}
e^{-\alpha^2/(4p\sigma^2)}
.
\end{aligned}
\end{equation*}
\end{proof}

\begin{remark}
The number 4 in the exponent in the statement of Lemma~\ref{lem:concentration_gauss_maxnorm}
may be reduced to $ 2 +\varepsilon $ 
for any $\varepsilon \in (0,2]$ by applying the H\"older inequality instead of the Cauchy--Schwarz inequality in the last step of the proof. The constant $C_p$ would then also depend on the choice of 
$\varepsilon$.
\end{remark}

Let $q\in [1,\infty)$.
Suppose for every $\varepsilon \in (0,1)$ there exists $H_R^*$, which can be seen as an approximation to $H_R$,
and there exists $\Psi^*$, which can be seen as an approximation to $\Psi$,
such that for every $x,p\in \bbR^d$, $t\in [0,t_f]$
\begin{equation}\label{eq:approx_H_Psi}
\max\{| H_R(t,x,p) - H_R^*(t,x,p)|, |\Psi(x) - \Psi^*(x) |\}
\leq 
\varepsilon 
(1+ \|x\|^q_2 + \|p\|^q_2)
.
\end{equation}
Moreover, there exists $C_{H_R,p}>0$ 
such that for every $x,p,p'\in \bbR^d$, $t\in [0,t_f]$
\begin{equation}\label{eq:Lipschitz_H_Hstar}
|H_R(t,x,p) - H_R(t,x,p')|
\leq 
C_{H_R,p} \|p-p'\|_\infty
.
\end{equation}
The constant $C_{H_R,p}$ does not depend 
on the dimension $d$.
The constants $C_{H_R,p}, C$ 
do not depend on $\varepsilon$.
Furthermore, assume that $H_R^*$
is globally Lipschitz continuous.
In particular, there exist constants $C_{H_R^*,x}, C_{H_R^*,p}$
such that for every $x,x',p,p'\in\bbR^d$,
\begin{equation}\label{eq:Lipschitz_H_R*}
    \sup_{t\in [0,t_f]}|H_R^*(t,x,p) - H_R^*(t,x',p')|
    \leq 
    C_{H_R^*,x} \|x-x'\|_1 + C_{H_R^*,p}\|p-p'\|_1
    .
\end{equation}
Also assume that $\Psi^*\in C^2(\bbR^d)$
and is globally Lipschitz continuous, i.e., 
there exists a constant $C_{\Psi^*,x}$
such that for every $x,x' \in\bbR^d$
\begin{equation}\label{eq:Lipschitz_Psi_R*}
    |\Psi^*(x) - \Psi^*(x')|
    \leq 
    C_{\Psi^*,x}
    \|x-x'\|_1
    .
\end{equation}
This implies in particular that $\Psi^*$
is at most linearly growing.
Denote by $V^*$ the solution to~\eqref{eq:HJB_PDE} 
with $H_R$ replaced by $H_R^*$
and $\Psi$ replaced by $\Psi^*$, which is 
in $C^{1,2}([0,t_f]\times \bbR^d)$ by 
Proposition~\ref{prop:appendix_classical_sol}.

The following lemma is inspired by~\cite[Lemma~2.3]{HJKN_2020}
and its proof.

\begin{lemma}
\label{lem:perturbation_value_function}
There exists a constant $C>0$ that does not depend on $d,\bar{d}$
such that for any $\delta\in (0,1)$
and for every $s\in [0,t_f)$, $x\in\bbR^d$
\begin{equation*}
\begin{aligned}
&\bbE( 
\|
(V(s,x + W_s), \nabla_x V(s,x+W_s))^\top
-
(V^*(s,x+W_s), \nabla_x V^*(s,x+W_s))^\top
\|_\infty
 )
 \\
&\leq 
\varepsilon^{1-\delta}
C e^{C_{H_R,p} (2(t_f-s) + \sqrt{t_f-s}3/(\sqrt{2}\delta) )}
\\
&\quad \times
\left(1+ \sup_{t\in[0,t_f]}\sup_{x',p'\in \bbR^d}|H_R(t,x',p')| + \|x\|^q_2 + t_f^2 + d^{q/2}(t_f C_{H_R^*,x} + C_{\Psi^*,x})^q + d^{3/2 +\delta} + \frac{1}{\sqrt{t_f-s}}\right)
.
\end{aligned}
\end{equation*}
\end{lemma}

\begin{proof}
By~\eqref{eq:stoch_repr_HJB} 
and by~\cite[Lemma~4.2(ii)]{HJK_2019} applied to $V^*$, 
for every $s\in [0,t_f)$, $x\in \bbR^d$, 
it holds that
\begin{equation}\label{eq:estimate1}
\begin{aligned}
(I)&:=
\|
(V(s,x), \nabla_x V(s,x))^\top
-
(V^*(s,x), \nabla_x V^*(s,x))^\top
\|_\infty
\\
&\leq 
\bbE
\left(
|\Psi(x+ W_{t_f -s}) - \Psi^*(x+ W_{t_f-s}) |
(1+\| \tfrac{W_{t_f-s}}{t_f-s}\|_\infty )
\right)
\\
&\quad+
\bbE\left( \int_{s}^{t_f}
|H_R(t,x+W_{t-s}, \nabla_x V(t,x+W_{t-s}))
\right.
\\
&\qquad\qquad\qquad
\left.
- H_R^*(t,x+W_{t-s}, \nabla_x V^*(t,x+W_{t-s}))|
(1+\| \tfrac{W_{t-s}}{t-s}\|_\infty)
{\rm d}t
\right)
.
\end{aligned}
\end{equation}
The estimation of the error is here complicated by the occurrence of the unbounded random weight 
$\| \tfrac{W_{t-s}}{t-s}\|_\infty$.
We re-organize the terms in~\eqref{eq:estimate1} to separately study large deviation of $\| \tfrac{W_{t-s}}{t-s}\|_\infty$
from its mean as follows. 
For any $\alpha>0$ to be chosen below,

\begin{equation*}
\begin{aligned}
(I)
&\leq 
\bbE
\left(
|\Psi(x+ W_{t_f -s}) - \Psi^*(x+ W_{t_f-s}) |
\left(1+\| \tfrac{W_{t_f-s}}{t_f-s}\|_\infty \right)
\right)
\\
&\quad +
\bbE\left( \int_{s}^{t_f}
|H_R(t,x+W_{t-s}, \nabla_x V(t,x+W_{t-s}))
- H_R^*(t,x+W_{t-s}, \nabla_x V^*(t,x+W_{t-s}))|
\right.
\\
&\quad\qquad\qquad
\left.
\times
\left(1+\| \tfrac{W_{t-s}}{t-s}\|_\infty \mathbbm{1}_{\left\{\|\tfrac{W_{t-s}}{t-s}\|_\infty > \sqrt{2\log(2d)/(t-s)} +\alpha \right\}}  \right)
{\rm d}t
\right)
\\
&\quad +
\bbE\left( \int_{s}^{t_f}
|H_R(t,x+W_{t-s}, \nabla_x V(t,x+W_{t-s}))
- H_R^*(t,x+W_{t-s}, \nabla_x V^*(t,x+W_{t-s}))|
\vphantom{\left(1+\frac{\sqrt{2\log(2d)}}{\sqrt{t-s}} +\alpha \right)}
\right.
\\
&\quad\qquad\qquad
\left.
\times
\left(1+\frac{\sqrt{2\log(2d)}}{\sqrt{t-s}} +\alpha \right)
{\rm d}t
\right).
\end{aligned}
\end{equation*}
It holds that the gradient of $V$ and $V^*$ are bounded with respect to $t$ and $x$. 
This is implied by~\eqref{eq:bound_grad_V} and by~\cite[Lemma~4.2(iv)]{HJK_2019}, 
respectively. 
Thus, by Lemma~\ref{lem:concentration_gauss_maxnorm} applied with $p=2$ and the Cauchy--Schwarz inequality, 
and~\eqref{eq:approx_H_Psi}
\begin{equation*}
\begin{aligned}
(I)
&\leq 
\bbE
\left(
|\Psi(x+ W_{t_f -s}) - \Psi^*(x+ W_{t_f-s}) |
\left(1+\| \tfrac{W_{t_f-s}}{t_f-s}\|_\infty \right)
\right)
\\
&\quad+
C\left(1+\sup_{t\in [0,t_f ]}\sup_{x',p'\in\bbR^d}\{| H_R(t,x',p') | + \| \nabla_x V^*(t,x')\|^q_2\}+  \sup_{t\in[s,t_f]}\bbE(\|x+W_{t-s})\|^{2q}_2)^{1/2}
\right)
\\
&\qquad\qquad
 \times 
\sqrt{\log(2d)} \int_{s}^{t_f}   \frac{e^{-\alpha^2(t-s)/8}}{\sqrt{t-s}}{\rm d}t
\\
&\quad+
\int_{s}^{t_f} \bbE\left( 
\left|H_R(t,x+W_{t-s}, \nabla_x V(t,x+W_{t-s}))
- H_R^*(t,x+W_{t-s}, \nabla_x V^*(t,x+W_{t-s}))\right| \right)
\\
&\qquad\qquad
\times\left(2+\frac{\sqrt{2\log(2d)}+\beta}{\sqrt{t-s}} \right)
{\rm d}t
.
\end{aligned}
\end{equation*}
We choose $\alpha = \beta/\sqrt{t-s}$
for some $\beta>0$ to be determined below
and obtain 
\begin{equation*}
    \begin{aligned}
(I)
&\leq
\bbE
\left(
|\Psi(x+ W_{t_f -s}) - \Psi^*(x+ W_{t_f-s}) |
\left(1+\| \tfrac{W_{t_f-s}}{t_f-s}\|_\infty \right)
\right)
\\
&\quad+
2C\left( 1+\sup_{t\in [0,t_f ]}\sup_{x',p'\in\bbR^d}\{| H_R(t,x',p') | + \| \nabla_x V^*(t,x')\|^q_2\} +  \sup_{t\in[s,t_f]}\bbE\left(\|x+W_{t-s})\|^{2q}_2\right)^{1/2} 
\right)
\\
&\qquad\qquad
\times
\sqrt{\log(2d)}  \sqrt{t_f-s}  e^{-\beta^2/8}
\\
&\quad+
\int_{s}^{t_f} \bbE\left( 
\left|H_R(t,x+W_{t-s}, \nabla_x V(t,x+W_{t-s}))
- H_R^*(t,x+W_{t-s}, \nabla_x V^*(t,x+W_{t-s}))\right| \right)
\\
&\qquad\qquad
\times\left(2+\frac{\sqrt{2\log(2d)}+\beta}{\sqrt{t-s}} \right)
{\rm d}t
,
\end{aligned}
\end{equation*}
where the constant $C$ does not depend on the dimension $d$.
Now, 
Fubini's theorem, 
the independent increment property of the Brownian motion $(W_t)_{t\in [0,t_f]}$
implies with~\eqref{eq:Lipschitz_H_Hstar}
as in the derivation of~\cite[Equation~(21)]{HJKN_2020}
that
\begin{equation}\label{eq:tech_lemma_main_step}
\begin{aligned}
(II):=
&\bbE( 
\|
(V(s,x + W_s), \nabla_x V(s,x+W_s))^\top
-
(V^*(s,x+W_s), \nabla_x V^*(s,x+W_s))^\top
\|_\infty
 )
 \\
 &\leq 
 \bbE
\left(
|\Psi(x+ W_{t_f }) - \Psi^*(x+ W_{t_f}) |
\left(1+\| \tfrac{W_{t_f-s}}{t_f-s}\|_\infty \right)
\right)
\\
&\quad+
2C(1+\sup_{t\in [0,t_f ]}\sup_{x',p'\in\bbR^d}\{| H_R(t,x',p') | + \| \nabla_x V^*(t,x')\|^q_2\}
+  \sup_{t\in[s,t_f]}\bbE(\|x+W_{t}\|^{2q})^{1/2} )
\\
&\quad\qquad
\times
\sqrt{\log(2d)}  \sqrt{t_f-s}  e^{-\beta^2/8}
\\
&\quad +
C_{H_R,p}
\int_{s}^{t_f} \bbE\left( 
\|  \nabla_x V(t,x+W_{t}) -  \nabla_x V^*(t,x+W_{t}) \|_\infty
\right)
\left(2+\frac{\sqrt{2\log(2d)}+\beta}{\sqrt{t-s}} \right)
{\rm d}t
\\
&\quad  
+
\sup_{t\in[0,t_f]}
 \bbE\left( 
\left|H_R(t,x+W_{t}, \nabla_x V^*(t,x+W_{t}))
- H_R^*(t,x+W_{t}, \nabla_x V^*(t,x+W_{t}))\right| \right)
\\
&\qquad\qquad
\times
\left(2(t_f -s)  + \left(2\sqrt{2\log(2d)}+2\beta\right) \sqrt{t_f -s}\right)
.
 \end{aligned}
 \end{equation}
 The Cauchy--Schwarz inequality, \eqref{eq:approx_H_Psi} and~\eqref{eq:lp_estimate_max_norm}
 imply that there exists a constant $C>0$ (only depending on $q$) that does not depend on $d$ 
 such that
 \begin{equation*}
 \bbE
\left(
|\Psi(x+ W_{t_f }) - \Psi^*(x+ W_{t_f}) |
\left(1+\| \tfrac{W_{t_f-s}}{t_f-s}\|_\infty \right)
\right)
\leq 
\varepsilon 
C \left(1 + \|x\|^q_2 +\sqrt{t_f d }\right)\left(1 + \frac{\sqrt{\log(2d)}}{\sqrt{t_f -s}}\right)
.
\end{equation*}
 Note that by~\cite[Lemma~4.2(iv)]{HJK_2019}, the gradient of $V^*$ is bounded, i.e., 
 \begin{equation*}
     \sup_{t\in [0,t_f]}\sup_{x'\in\bbR^d}\|\nabla_x V^*(t,x')\|_2
     \leq 
     \sqrt{d}(t_f C_{H_R^*,x} + C_{\Psi^*,x}).
 \end{equation*}
 The boundedness of the gradient of $V^*$, the  
 Cauchy--Schwarz inequality,  and~\eqref{eq:approx_H_Psi} imply
 that there exists a constant $C>0$ that only depends on $q$
 such that
 \begin{equation*}
 \begin{aligned}
 &\sup_{t\in [0,t_f]}\bbE\left( 
|H_R(t,x+W_{t}, \nabla_x V^*(t,x+W_{t}))
- H_R^*(t,x+W_{t}, \nabla_x V^*(t,x+W_{t}))| \right)
\\
&\qquad\leq 
\varepsilon 
C \left(1 + \|x\|^q_2 +\sqrt{t_f}\sqrt{d} + 
d^{q/2}(t_f C_{H_R^*,x} + C_{\Psi^*,x})^q
\right)
.
\end{aligned}
 \end{equation*}
 Thus, by~\eqref{eq:tech_lemma_main_step}
 there exists a constant $C>0$ 
 that does not depend on $d$ such that
 \begin{equation*}
 \begin{aligned}
 (II)
 &\leq
 C(1+ \sup_{t\in [0,t_f]}\sup_{x',p'\in\bbR^d}\{|H_R(t,x',p')|\} +\|x\|^q_2 + \sqrt{t_f d} + d^{q/2}(t_f C_{H_R^*,x} + C_{\Psi^*,x})^q)
(1+\sqrt{\log(d)} + \beta)
\\
&\qquad
\times
\left(1+(t_f-s)  + \frac{1}{\sqrt{t_f -s }}\right)  (e^{-\beta^2/8} + \varepsilon)
\\
&\quad +
C_{H_R,p}
\int_{s}^{t_f} \bbE\left( 
\|  \nabla_x V(t,x+W_{t}) -  \nabla_x V^*(t,x+W_{t}) \|_\infty
\right)
\left(2+\frac{\sqrt{2\log(2d)}+\beta}{\sqrt{t-s}} \right)
{\rm d}t
.
 \end{aligned}
 \end{equation*}
 The Gronwall inequality implies that 
 \begin{equation*}
 \begin{aligned}
 (II)
 &\leq 
 C(1+ \sup_{t\in [0,t_f]}\sup_{x',p'\in\bbR^d}\{|H_R(t,x',p')|\} + \|x\|^q_2 + \sqrt{t_f d} + d^{q/2}(t_f C_{H_R^*,x} + C_{\Psi^*,x})^q)
(1+\sqrt{\log(d)} + \beta)
\\
&\quad 
\times\left(1+(t_f-s)  + \frac{1}{\sqrt{t_f -s }}\right)  (e^{-\beta^2/8} + \varepsilon)
 \\
 &\quad \times
 \exp\left( 
 C_{H_R,p}
 2(t_f -s) + 2C_{H_R,p}\left(\sqrt{2\log(2d)} +\beta\right)\sqrt{t_f -s}
 \right) .
 \end{aligned}
 \end{equation*}
 The assertion follows with the choice $\beta = \sqrt{8\log(\varepsilon^{-1})}$,
 where we used the
 fact that
 $\exp(\sqrt{c_0 \log(x)  }) \leq \exp(\sqrt{c_0}/(4\delta)) x^\delta$
 for every $c_0,\delta >0$ and every $x\in [1,\infty)$.
\end{proof}

\section{DNN approximation of the Hamiltonian}
\label{sec:DNN_approx_H}

The Hamiltonian shall be approximated by DNNs with explicit bounds on the Lipschitz constants 
of these DNN approximations.
We shall begin by collecting
some elementary properties of certain ReLU DNNs, 
which will serve us as building blocks in our analysis.

\begin{lemma}[Proposition~2 in \cite{yarotsky_2017}]
\label{lem:DNN_sq}
For every $\varepsilon\in (0,1)$, 
there exists a ReLU DNN $\phi_{{\rm sq},\varepsilon}$
such that for every $x\in [0,1]$
\begin{equation*}
    |x^2 - \phi_{{\rm sq},\varepsilon}(x)|
    \leq \varepsilon
\end{equation*}
and ${\rm size}(\phi_{{\rm sq},\varepsilon})=\calO(\lceil \log(\varepsilon^{-1})\rceil)$.
\end{lemma}

The following lemma is a slight extension of~\cite[Proposition~3]{yarotsky_2017}

\begin{lemma}\label{lem:ReLU_Prod_Lip_bound}
For every $M>0$ and $\delta>0$, there exists a ReLU DNN
such that
for every $x,y\in [-M,M]$
\begin{equation*}
    |xy - \tilde{\times}_\delta(x,y) |
    \leq 
    \delta 
\end{equation*}
and ${\rm size}(\tilde{\times}_\delta)= \calO(\lceil\log(\delta^{-1})\rceil )$. 
Moreover, for 
every $x,x'y,y' \in [-M,M]$
\begin{equation*}
    | \tilde{\times}_\delta(x,y) - \tilde{\times}_\delta(x',y')|
    \leq 
    4M (|x-x'| + |y-y'|)
    .
\end{equation*}
\end{lemma}

\begin{proof}
 By Lemma~\ref{lem:DNN_sq}, 
 for every $\varepsilon\in (0,1)$,
 there exists a ReLU DNN $\phi_{\rm sq,\varepsilon}$ with uniformly (with respect to $\varepsilon$) bounded weights 
 such that 
 \begin{equation}\label{eq:bound_DNN_sq}
     |x^2 - \phi_{{\rm sq},\varepsilon}(x)|
     \leq 
     \varepsilon
     \quad \forall x \in [0,1]
 \end{equation}
 and ${\rm size}(\phi_{{\rm sq}, \varepsilon})=\calO(\lceil \log(\varepsilon^{-1}) \rceil)$.
 Since the mapping $[0,1]\ni x\to \phi_{{\rm sq},\varepsilon}(x)$
 equals the linear interpolant of the mapping $[0,1]\ni x \to x^2$
 on an equispaced grid for every $\varepsilon\in (0,1)$,
 the Lipschitz constant 
 of the ReLU DNN $\phi_{{\rm sq},\varepsilon}$ is upper bounded by $2$ for every $\varepsilon \in (0,1)$.
 For the approximation of the product we apply as in the proof of~\cite[Proposition~3]{yarotsky_2017}
 the following formular and subsequently replace the square by $\phi_{{\rm sq},\varepsilon}$. 
 For any $M>0$ and any $x,y \in \bbR$,
  $
     xy 
     = 
     2M^2
     [ 
     (\tfrac{|x+y|}{2M} )^2
     -(\tfrac{|x|}{2M})^2
     - (\tfrac{|y}{2M})^2
     ]
     $.
 Thus, for any $M>0$ there exists a DNN $\tilde{\times}_{\delta}$ 
 such that 
 \begin{equation*}
     \tilde{\times}_\delta(x,y)
     =
     2 M^2
     \left[ 
     \phi_{{\rm sq},\varepsilon}\left(\frac{|x+y|}{2M} \right)
     -\phi_{{\rm sq},\varepsilon}\left(\frac{|x|}{2M}\right)
     - \phi_{{\rm sq},\varepsilon}\left(\frac{|y|}{2M}\right)
     \right].
 \end{equation*}
 By~\eqref{eq:bound_DNN_sq}, for any $x,y\in\bbR$
 \begin{equation*}
     |xy - \tilde{\times}_\delta(x,y)|
     \leq 
     2M^2 \varepsilon
     ,
 \end{equation*}
 which implies the first assertion with the choice $\varepsilon = \delta /(2M^2)$.
 The Lipschitz continuity of $\phi_{{\rm sg},\varepsilon}$ implies
 that 
 \begin{equation*}
     \begin{aligned}
     |\tilde{\times}_\delta(x,y)
     -
     \tilde{\times}_\delta(x',y')|
     &\leq 
     2M^2\left[  
     \left|\phi_{{\rm sq},\varepsilon}\left(\frac{|x+y|}{2M} \right) - \phi_{{\rm sq},\varepsilon}\left(\frac{|x'+y'|}{2M} \right) \right|
     + 
     \left|\phi_{{\rm sq},\varepsilon}\left(\frac{|x|}{2M}\right) -\phi_{{\rm sq},\varepsilon}\left(\frac{|x'|}{2M}\right)\right|
     \right.
     \\
     &
     \qquad+\left.
     \left|\phi_{{\rm sq},\varepsilon}\left(\frac{|y|}{2M}\right)- \phi_{{\rm sq},\varepsilon}\left(\frac{|y'|}{2M}\right)\right|
     \right]
     \\
     &\leq
     2M
     \left( 
     ||x+y| - |x'+y'||
     + ||x| - |x'|| + ||y| -|y'||
     \right)
     \\
     &\leq 
     4M 
     (|x-x'| + |y-y'|)
     .
     \end{aligned}
 \end{equation*}
\end{proof}

\begin{lemma}\label{lem:ReLU_Matr_vect_prod}
For every $M>0$ and $\delta\in (0,1)$,
there exists a ReLU DNN $\tilde{\times}_{\delta}$
such that for every $A\in [-M,M]^{m\times n}$, every
$b\in [-M,M]^{n}$, $p\in [1,\infty]$
 \begin{equation*}
      \|Ab -  \tilde{\times}_\delta(A,b)  \|_p
     \leq \delta m^{1/p}n 
 \end{equation*}
and ${\rm size}(\tilde{\times}_\delta) = 
\calO((\lceil\log(\delta^{-1})\rceil mn)$
with obvious modifications in the case $p=\infty$.
Moreover, for every $A,A'\in [-M,M]^{m\times n}$, every
$b,b'\in [-M,M]^{n}$, $p\in [1,\infty]$,
\begin{equation*}
    \begin{aligned}
    \| \tilde{\times}_\varepsilon(A,b) - \tilde{\times}_\varepsilon(A',b')\|_p
    &\leq 4M \left(
    \sum_{i=1}^m
    \left( 
    \sum_{j=1}^n
    |A_{ij}  - A'_{ij}|
    +
    |b_j - b'_j|
    \right)^p
    \right)^{1/p}
    \\
    &\leq 
    4M (2n)^{((p-1)/p)}
    (\|A - A'\|_p
    +
    m^{1/p}\|b-b'\|_p)
    \end{aligned}
\end{equation*}
with obvious modifications in the case $p=\infty$.
\end{lemma}

\begin{proof}
The assertion follows by a component-wise application of
Lemma~\ref{lem:ReLU_Prod_Lip_bound} and the H\"older inequality.
\end{proof}

\begin{proposition}\label{prop:DNN_approx_H}
Let $q\in [1,\infty)$.
 Suppose for every $\delta_1,\delta_2,\delta_{\bar{L}}\in (0,1)$, there exist DNNs $\phi_{f_1,\delta_1}$,
 $\phi_{f_2,\delta_2}$, $\phi_{\bar{L},\delta_{\bar{L}}}$
 such that for every $t\in [0,t_f]$
 and every $x\in \bbR^d$
 \begin{equation*}
     \|f_{1}(t,x) - \phi_{f_1,\delta_1}(t,x)\|_2,
     \leq \delta_1(1 + \|x\|_2^q),
 \end{equation*}
 \begin{equation*}
     \|f_{2}(t,x) - \phi_{f_2,\delta_2}(t,x)\|_2 \leq \delta_2(1 + \|x\|_2^q),
 \end{equation*}
 and 
 \begin{equation*}
     |\bar{L}(t,x) - \phi_{\bar{L},\delta_{\bar{L}}}(t,x)| \leq \delta_{\bar{L}}(1 + \|x\|_2^q).
 \end{equation*}
 Suppose that there exist $C,\kappa>0$ 
 such that for every $d,\bar{d}\in\bbN$,
 \begin{equation*}
     \sup_{t\in [0,t_f]}\sup_{x\in\bbR^d}\{\|\phi_{f_1,\delta_1}(t,x)\|_\infty
     +
     \|\phi_{f_2,\delta_2}(t,x)\|_\infty
     +\|\phi_{\bar{L},\delta_{\bar{L}}}(t,x)\|_\infty\}
     \leq C d^\kappa \bar{d}^\kappa
     .
 \end{equation*}
 Moreover, we assume that ${\rm  size}(\phi_{f_1,\delta_1})= \calO( d^{\kappa_1} \delta_1^{-\kappa_1})$,
 ${\rm  size}(\phi_{f_2,\delta_2})= \calO( d^{\kappa_2} \delta_2^{-\kappa_2})$, 
 and 
 ${\rm  size}(\phi_{\bar{L},\delta_{\bar{L}}})= \calO( d^{\kappa_L} \delta_{\bar{L}}^{-\kappa_L})$. 
 Then, for every $\delta \in (0,1)$, 
 there exists a DNN $\phi_{H,\delta}$
 such that for every $x\in\bbR^d$,
 every $p\in [-R,R]^d$, and every $t\in [0,t_f]$
 \begin{equation*}
     |H(t,x,p) - \phi_{H,\delta}(t,x,p)| \leq \delta(1 + \|x\|^{q}_2  ).
 \end{equation*}
 Moreover, it holds that ${\rm size } (\phi_{H,\delta}) = \calO(d^{\max\{ \kappa_1,\kappa_2,\kappa_{\bar{L}}   \}+ \kappa_2 } \bar{d}^{\kappa_2} \delta^{-\max\{ \kappa_1,\kappa_2,\kappa_{\bar{L}}   \}} + d \bar{d}\log(d \bar{d}) \lceil\log(\delta^{-1})\rceil) $.
\end{proposition}

\begin{proof}
 There exists a DNN $\phi_{H,\delta}$
 such that
 \begin{equation}
 \label{eq:def_DNN_H}
 \begin{aligned}
 \phi_{H,\delta}(t,x,p)
 &=
 \tilde{\times}_{\delta_{\rm prod}}(p^\top,
 \phi_{f_1,\delta_1}(t,x) + 
 \tilde{\times}_{\delta_{\rm prod}}(\phi_{f_2,\delta_2}(t,x),\tilde{u}(t,x,p)))
 \\
 &\quad+ 
 \phi_{\bar{L},\delta_{\bar{L}}}(t,x) + \gamma
 \sum_{i=1}^{\bar{d}}
 \max\{|a_i|^2, |b_i|^2\}
 \phi_{{\rm sq},\delta_{\rm sq}}\left(\frac{(\tilde{u}(t,x,p))_i}{\max\{|a_i|,|b_i|\}}\right)
 ,
 \end{aligned}
 \end{equation}
 where the coordinates of $\tilde{u}$ are given by
 \begin{equation*}
 (\tilde{u}(t,x,p))_i
    =
    \min\{ \max\{- (\tilde{\times}_{\delta_{\rm prod}}(\phi_{f_2,\delta_2}(t,x)^\top, p) )_i /(2\gamma)  , a_i\}, b_i\}, 
    \quad 
    i=1,\ldots, \bar{d}
    .
 \end{equation*}
 The DNNs $\tilde{\times}_{\rm prod}$
 and $\phi_{{\rm sq},\delta_{\rm sq}}$
 are applied with values from bounded sets, 
 we shall not indicate this in our notation.
 The expression of $\phi_{H,\delta}$
 has three terms. The second does not pose problems in the analysis of the error. 
 The first and third term include compositions of several DNNs. 
 These need to be treated in detail.
 
 The fact that for any $y_1,y_2,c,d \in \bbR $
 it holds that 
 \begin{equation}\label{eq:Lipschitz_min_max}
 |\min\{\max\{y_1,c\},d\} - \min\{\max\{y_2,c\},d\}|
 \leq 
 \min\{\max\{2|c|,2|d|\}, |y_1-y_2|\}
 \end{equation}
 implies with Lemma~\ref{lem:formula_H} that
 \begin{equation*}
 |u_i(t,x,p) - \tilde{u}_i(t,x,p)|
 \leq 
 \min\{\max\{2|a_i|, 2|b_i|\}, | (f_2(t,x)^\top p)_i - (\tilde{\times}_{\delta_{\rm prod}}(\phi_{f_2,\delta_2}(t,x)^\top, p) )_i| /(2\gamma)     \} 
 .
 \end{equation*}
By Lemma~\ref{lem:ReLU_Matr_vect_prod} it holds that 
\begin{equation}\label{eq:bounde_f2_phi2}
\begin{aligned}
&| (f_2(t,x)^\top p)_i - (\tilde{\times}_\delta(\phi_{f_2,\delta_2}(t,x)^\top, p) )_i|
\\
&\quad\leq 
|((f_2(t,x)  -   \phi_{f_2,\delta_2}(t,x))^\top p)_i|
+
|(\phi_{f_2,\delta}(t,x)^\top p)_i -  (\tilde{\times}_\delta(\phi_{f_2,\delta_2}(t,x)^\top, p) )_i|
\\
&\quad\leq 
\|f_2(t,x)  -   \phi_{f_2,\delta_2}(t,x)\|_2\|p\|_2
+
\delta_{\rm prod}
d
\\
&\quad \leq 
\delta_2(1+\|x\|_2^q )\|p\|_2
+
\delta_{\rm prod}
d
.
\end{aligned}
\end{equation}
In conclusion, it holds that 
\begin{equation}\label{eq:estimate_0}
\begin{aligned}
&|u_i(t,x,p) - \tilde{u}_i(t,x,p)|
 \\
 &\quad \leq 
 \min\{\max\{2|a_i|, 2|b_i|\},
  \delta_2(2\gamma)^{-1}(1+\|x\|_2^q )\|p\|_2
+
\delta_{\rm prod}(2\gamma)^{-1}
d
.
\end{aligned}
\end{equation}
We obtain the following bound for the error related to the forth term
in~\eqref{eq:def_DNN_H}
\begin{equation*}
\begin{aligned}
&\left|\|u(t,x,p)\|^2_2 
-
\sum_{i=1}^{\bar{d}}
\max\{|a_i|, |b_i|\}
\phi_{{\rm sq},\delta_{\rm sq}}
\left(
\frac{\tilde{u}_i}{\max\{|a_i|,|b_i|\}}
\right)\right|
\\
&\qquad\leq 
|\|u(t,x,p)\|^2_2  - \|\tilde{u}_i\|_2^2   |
+ 
\sum_{i=1}^{\bar{d}}
\left|\tilde{u}_i^2 - \max\{|a_i|, |b_i|\}
\phi_{{\rm sq},\delta_{\rm sq}}
\left(
\frac{\tilde{u}_i}{\max\{|a_i|,|b_i|\}}
\right) \right|
\\
&\qquad\leq 
2 \sum_{i=1}^{\bar{d}}
\max\{|a_i|, |b_i|\}
|u_i(t,x,p)  - \tilde{u}_i   |
+ \bar{d}\max\{\|a\|_\infty,\|b\|_\infty\} \delta_{\rm sq}
\\
&\qquad\leq 
2 
\bar{d}\max\{\|a\|_\infty,\|b\|_\infty\}
[
  \delta_2(2\gamma)^{-1}(1+\|x\|_2^q )\|p\|_2
+
\delta_{\rm prod}(2\gamma)^{-1}
d
  \\
&\quad\qquad
+\bar{d}\max\{\|a\|_\infty,\|b\|_\infty\} \delta_{\rm sq}
\\ 
&\qquad\leq 
\max\{\delta_2,\delta_{\rm prod}, \delta_{\rm sq}\}
C
\bar{d}\max\{\|a\|_\infty,\|b\|_\infty\}
d
,
\end{aligned}
\end{equation*}
where the constant $C>0$ does not depend on $\bar{d},d,a,b$. 

By Lemma~\ref{lem:ReLU_Matr_vect_prod} and elementary manipulations, 
we estimate the second term on the right hand side of~\eqref{eq:def_DNN_H}
as follows
\begin{equation}\label{eq:estimate_1}
    \begin{aligned}
    &| p^\top f_2(t,x) u(t,x,p)
    -
    \tilde{\times}_{\delta_{\rm prod}}(p^\top, \tilde{\times}_{\delta_{\rm prod}}(\phi_{f_2,\delta_2}(t,x),\tilde{u}(t,x,p)))|
    \\
    &\quad \leq 
    | p^\top (f_2(t,x) u(t,x,p) -
    \phi_{f_2,\delta_2}(t,x) u(t,x,p)
    )|
    \\
    &\qquad + 
    |p^\top \phi_{f_2,\delta_2}(t,x) (u(t,x,p)- \tilde{u}(t,x,p)    )|
    \\
    &\qquad+
    |p^\top (\phi_{f_2,\delta_2}(t,x) \tilde{u}(t,x,p)
    -
    \tilde{\times}_{\delta_{\rm prod}}(\phi_{f_2,\delta_2}(t,x),\tilde{u}(t,x,p)))  | \\
    &\qquad+
    |p^\top \tilde{\times}_{\delta_{\rm prod}}(\phi_{f_2,\delta_2}(t,x),\tilde{u}(t,x,p))
    - \tilde{\times}_{\delta_{\rm prod}}(p^\top, \tilde{\times}_{\delta_{\rm prod}}(\phi_{f_2,\delta_2}(t,x),\tilde{u}(t,x,p)))| 
    \\
    &\quad\leq 
     \delta_2(1+ \|x\|_2^q)\|p\|_2 \max\{\|a\|_\infty,\|b\|_\infty\} 
     + 
     \|p\|_2 \| \phi_{f_2,\delta_2}(t,x)\|_2 \|u(t,x,p)- \tilde{u}(t,x,p) \|_2
     \\
     &\qquad+ 
     \|p\|_2 \delta_{\rm prod}
     d^{1/2}\bar{d}
     + 
     \delta_{\rm prod} 
     d
     .
    \end{aligned}
\end{equation}
The estimates in~\eqref{eq:estimate_0}
and \eqref{eq:estimate_1}
result in
\begin{equation*}
    \begin{aligned}
    &| p^\top f_2(t,x) u(t,x,p)
    -
    \tilde{\times}_{\delta_{\rm prod}}(p^\top, \tilde{\times}_{\delta_{\rm prod}}(\phi_{f_2,\delta_2}(t,x),\tilde{u}(t,x,p)))|
    \\
    &\quad\leq 
    C \max\{\delta_2,\delta_{\rm prod}\}
    (d+d^{1/2} \bar{d} + \|x\|_2^q)
    .
    \end{aligned}
\end{equation*}
The assertion now follows upon a corresponding bound for 
$|p^\top (f_1 -  \phi_{f_1,\delta_1})|$ (see~\eqref{eq:bounde_f2_phi2}) with assumed approximation properties of $\phi_{f_1,\delta_1}$, $\phi_{f_2,\delta_2}$, and $\phi_{\bar{L}, \delta_{\bar{L}}}$
with the choices 
$ \delta_1\simeq \delta_2  \simeq \delta_{\rm prod} \simeq \delta_{\rm sq }
\simeq \delta /(d+d^{1/2} \bar{d})$ 
and $ \delta_{\bar{L}} \simeq \delta$.
\end{proof}

In the following, we study the Lipschitz continuity of DNN approximations 
of the Hamiltonian. 
We shall remark that the function $\chi_R$ may be represented exactly by a ReLU DNN. 

\begin{remark}\label{rmk:clip_RelU_DNN}
For every $R'>0$, there exists a ReLU DNN $\phi$ with size that does not depend on $R'$ 
such that $\chi_{R'}(y) = \phi(y)$ for every $y\in \bbR$, cf.~\cite[Lemma~A.1]{BGJ_2018}.
\end{remark}

\begin{proposition}\label{prop:DNN_Lipschitz}
 Let the assumptions of Proposition~\ref{prop:DNN_approx_H}
 be satisfied.
 Let the DNNs $\phi_{f_1,\delta}$, $\phi_{f_2,\delta}$, 
 $\phi_{\bar{L},\delta_{\bar{L}}}$ be globally Lipschitz continuous with respect to the second variable and the $1$-norm
 with corresponding
 Lipschitz constants 
 $C_{1,x}$, $C_{2,x}$, $C_{\bar{L},x}$.
 Then the DNN satisfying 
 \begin{equation}\label{eq:DNN_phi_H_R}
 \begin{aligned}
 \phi_{H,\delta,R}(t,x,p)
 &=
 \tilde{\times}_{\delta_{\rm prod}}(\chi_R(p)^\top,
 \phi_{f_1,\delta_1}(t,x) + 
 \tilde{\times}_{\delta_{\rm prod}}(\phi_{f_2,\delta_2}(t,x),\tilde{u}_R(t,x,p)))
 \\
 &\quad+ 
 \phi_{\bar{L},\delta_{\bar{L}}}(t,x) + \gamma
 \sum_{i=1}^{\bar{d}}
 \max\{|a_i|^2, |b_i|^2\}
 \phi_{{\rm sq},\delta_{\rm sq}}\left(\frac{(\tilde{u}_R(t,x,p))_i}{\max\{|a_i|,|b_i|\}}\right)
 ,
 \end{aligned}
 \end{equation}
 where the coordinates of $\tilde{u}_R$ are given by
 \begin{equation*}
 (\tilde{u}_R(t,x,p))_i
    =
    \min\{ \max\{- (\tilde{\times}_{\delta_{\rm prod}}(\phi_{f_2,\delta_2}(t,x)^\top, \chi_R(p)) )_i /(2\gamma)  , a_i\}, b_i\} 
    \quad 
    i=1,\ldots, \bar{d}
    ,
 \end{equation*}
 satisfies
 for every $t\in [0,t_f]$, $x,x',p,p' \in\bbR^d$,
 \begin{equation*}
 \begin{aligned}
     &|\phi_{H,\delta,R}(t,x,p) - \phi_{H,\delta,R}(t,x',p')   |
     \\
     &\quad\leq 
    \bar{C} 
    \left(
     (1+C_{\bar{L},x} + R C_{1,x} + C_{2,x} )\|x-x'\|_1  
     + \|p-p'\|_1
     \right)
     \end{aligned}
 \end{equation*}
 with 
 \begin{equation*}
     \bar{C}
     =
     \widetilde{C}
    d\bar{d}
    (1+R)^2
    \left(1+\sup_{t\in [0,t_f],x\in\bbR^d}\{\|\phi_{f_1,\delta_1}(t,x)\|_\infty +\|\phi_{f_2,\delta_2}(t,x)\|_\infty \}\right)^3,
 \end{equation*}
 where the constant $\widetilde{C}$ is a generic constant
 independent of $d$, $\bar{d}$, $C_{1,x}$, $C_{2,x}$, $C_{\bar{L},x}$, $R$
 that depends on $\gamma$, $\|a\|_\infty$, 
and $\|b\|_\infty$.
\end{proposition}

\begin{proof}
Since the maximum and the minimum and also the function $\chi_R$ may be represented by ReLU DNNs, cf.~Remark~\ref{rmk:clip_RelU_DNN}, 
there exists a DNN satisfying~\eqref{eq:DNN_phi_H_R}.
 Our main tool in this proof will be the simple observation that the Lipschitz constant of the composition of two functions is the product of 
 the individual Lipschitz constants,
 which follows readily.
 Let $F,G$ be two compatiple Lipschitz mappings, 
 then for every $x,y$ in the domain of $G$
 it holds (formally) that 
 \begin{equation}\label{eq:tech_trick_Lip_const}
     \| F(G(x)) - F(G(y))\|
     \leq 
     C_F \|G(x) - G(y)\|
     \leq 
     C_F C_G \|x - y\|, 
 \end{equation}
where $C_F,C_G$ denote the respective Lipschitz constants and $\|\cdot\|$ shall denote a (not necessarily the same)
appropriate norm.

The Lipschitz constant of the coordinates of $\tilde{u}_R$
with respect to the third argument $p$ is
$$\frac{1}{2\gamma}4\max\left\{R,\sup_{t\in [0,t_f], x\in\bbR^d} \|\phi_{f_2,\delta_2}(t,x)\|_\infty\right\}, $$ 
which follows by~\eqref{eq:tech_trick_Lip_const}, Lemma~\ref{lem:ReLU_Matr_vect_prod},
and the observation that the Lipschitz constant of $\chi_R$ and the maximum and mininum is bounded by one.
Thus, the Lipschitz constant of $\tilde{u}_R$ with respect to $p$ and 
the $1$-norm is 
\begin{equation*}
    C_{\tilde{u}_R,p}
    :=
    \frac{1}{\gamma}\bar{d}2\max\left\{R,\sup_{x\in\bbR^d} \|\phi_{f_2,\delta_2}(x)\|_\infty\right\}
    .
\end{equation*}
Similarly, it holds that the Lipschitz constant of $\tilde{u}_R$ with respect to $x$ and 
the $1$-norm is 
\begin{equation*}
    C_{\tilde{u}_R,x}
    :=
    \frac{1}{\gamma}\bar{d} C_{2,x}2\max\left\{R,\sup_{x\in\bbR^d} \|\phi_{f_2,\delta_2}(x)\|_\infty\right\}
    .
\end{equation*}
The expression of the DNN approximation $\phi_{H,\delta,R}$
has three terms. The respective Lipschitz contants are additive. 
We shall begin by analyzing the Lipschitz constant with respect to the third variable $p$.
By Lemma~\ref{lem:ReLU_Matr_vect_prod} and~\eqref{eq:tech_trick_Lip_const}, 
the  Lipschitz constant of $\phi_{H,\delta,R}$ with respect to $p$
and the $1$-norm
is
\begin{equation*}
\begin{aligned}
C_{\phi_{H,\delta,R},p}
&\leq 
    d C_{\tilde{u}_R,p} 16 \max\{\sup_{x\in\bbR^d}\|\phi_{f_2,\delta_2}(x)\|_\infty,\|a\|_\infty,\|b\|_\infty\}
    \\
    &\quad\times
    \max\left\{R,\sup_{x\in\bbR^d}\|\phi_{f_1,\delta_1}(x)\|_\infty+\sup_{x\in\bbR^d}\|\phi_{f_2,\delta_2}(x)\|_\infty\max\{\|a\|_\infty,\|b\|_\infty\}\right\}
    \\
    &\quad+
    2 C_{\tilde{u}_R,p}\gamma \max\{\|a\|_\infty,\|b\|_\infty\}
    \\
    &\leq 
    \widetilde{C}
    d\bar{d}
    (1+R^2)\left(1+\sup_{t\in [0,t_f],x\in\bbR^d}\{\|\phi_{f_1,\delta_1}(t,x)\|_\infty +\|\phi_{f_2,\delta_2}(t,x)\|_\infty \}\right)^3
    ,
\end{aligned}
\end{equation*}
where $\widetilde{C}>0$
only depends on $\gamma,\|a\|_\infty,\|b\|_\infty$.
By the same reasoning, the Lipschitz constant of $\phi_{H,\delta,R}$ 
with respect to $x$ denoted by $C_{\phi_{H,\delta,R},x}$
satisfies
\begin{equation*}
  C_{\phi_{H,\delta,R},x}
  \leq 
  \widetilde{C}d \bar{d}
  (1+R^2)
  (1+C_{\bar{L},x} + C_{1,x} + C_{2,x} )
  \left(1+\sup_{t\in [0,t_f],x\in\bbR^d}\{\|\phi_{f_1,\delta_1}(t,x)\|_\infty +\|\phi_{f_2,\delta_2}(t,x)\|_\infty \}\right)^3,
\end{equation*}
where again $\widetilde{C}>0$ only depends on $\gamma,\|a\|_\infty,\|b\|_\infty$.
\end{proof}

\begin{remark}\label{rmk:DNN_phi_H_R}
The DNN satisfying~\eqref{eq:DNN_phi_H_R}
also satisfies the approximation estimate and the size estimate in 
Proposition~\ref{prop:DNN_approx_H} with $p\in\bbR^d$
instead.
\end{remark}

\begin{corollary}\label{cor:bound_phi_H}
Let the assumptions of Proposition~\ref{prop:DNN_Lipschitz}
and Assumption~\ref{ass:lineardynamics}
be satisfied. 
Then, 
there exist constants
$C,\kappa>0$
such that for every $d,\bar{d}\in\bbN$
\begin{equation*}
    \sup_{\delta\in (0,1)}\sup_{t\in [0,t_f] ,x,p\in \bbR^d}
    |\phi_{H,\delta,R}|
    \leq 
    C d^\kappa 
    \bar{d}^\kappa
\end{equation*}
\end{corollary}

\begin{proof}
The assertion follows by Lemma~\ref{lem:ReLU_Matr_vect_prod} and~\eqref{eq:bound_R}.
\end{proof}

\section{DNN approximation of the value function}
\label{sec:DNN_approx_Value_fct}

In this section we will establish the approximation of the value function and its gradient 
by DNNs. 
There exist DNNs $\phi_{H,\delta_H,R}$, $\delta_H\in(0,1)$,
that approximate the Hamiltonian $H_R$ 
and are globally Lipschitz continuous and bounded according to Corollary~\ref{cor:bound_phi_H},
see Propositions~\ref{prop:DNN_approx_H}
and~\ref{prop:DNN_Lipschitz}.
Recall that $R$ is specified in~\eqref{eq:def_R} and an upper bound is given in~\eqref{eq:bound_R}.
We suppose that there exist globally Lipschitz ReCU DNN approximations $\phi_{\Psi,\delta_\Psi}$, $\delta_\Psi\in (0,1)$, 
to the terminal condition
$\Psi$. 
Specifically, there exists a constant $C_{\phi_{\Psi},x}>0$
such that for every $\delta_\Psi\in(0,1)$ and every $x,x'\in \bbR^d$
\begin{equation}\label{eq:Lipschitz_DNN_Psi}
    |\phi_{\Psi, \delta_\Psi}(x) - \phi_{\Psi, \delta_\Psi}(x')|
    \leq 
    C_{\phi_{\Psi},x}
    \|x-x'\|_1
    .
\end{equation}
The property that every ReCU DNN is twice continuously 
differentiable will be useful in the ensuing proof.
We shall denote the solution to the HJB equation~\eqref{eq:HJB_PDE} with respect to 
the DNN approximation $\phi_{H,\delta_H,R}$ of the Hamiltonian and with respect to the 
terminal condition $\phi_{\Psi,\delta_\Psi}$
by
$\widetilde{V}$, 
i.e., 
\begin{equation}\label{eq:HJB_PDE_DNN_H}
\partial_t \widetilde{V}(t,x)
 + \frac{1}{2}\Delta \widetilde{V}(t,x) + \phi_{H,\delta_H,R}(t,x,\nabla_x \widetilde{V}(t,x))=0
  \text{ on } [0,t_f)\times \bbR^d, 
  \quad \widetilde{V}(t_f,x) = \phi_{\Psi,\delta_\Psi}(x) \text{ on } \bbR^d
  .
\end{equation}
Proposition~\ref{prop:appendix_classical_sol} implies that 
$\widetilde{V} \in C^{1,2}([0,t_f]\times \bbR^d)$.

\begin{proposition}\label{prop:MLP_approx}
Let Assumption~\ref{ass:lineardynamics} and
 let the assumptions of 
 Proposition~\ref{prop:DNN_Lipschitz}
 be satisfied.
 Suppose there exist $C,\kappa_1,\kappa_2>0$ such that for every $d,\bar{d}\in\bbN$, 
 \begin{equation*}
 \sup_{\delta_\Psi\in(0,1)}\sup_{x\in\bbR^d}|\phi_{\Psi,\delta_\Psi}(x)|+ C_{\phi_{\Psi},x} + C_{1,x}+ C_{2,x}+ C_{\bar{L},x} \leq C d^{\kappa_1} \bar{d}^{\kappa_2}.
 \end{equation*}
 Let $\alpha\in(0,1)$.
 For every $N,M\in\bbN$, 
 define the index set $\Theta^N_N$ by $\Theta^N_0 = \{0\}$ and 
 \begin{equation}\label{eq:def_set_Theta}
  \Theta^N_k 
  =
  \Theta^N_{k-1} 
  \cup
  \left[
  \bigcup_{\stackrel{l=-k+1,\ldots,k-1}{i=-M^k,\ldots,M^k}}
  \{(\theta,l,i) : \theta \in \Theta^N_{k-1}\}\right]
  \quad 
  k=1,\ldots,N.
 \end{equation}
 Let $Z^\theta$, $\theta\in\Theta^N_N$, be i.i.d.\ $d$-dimensional, standard normally distributed 
 random variables, let $\tau^\theta$, $\theta\in\Theta^N_N$, i.i.d. $(0,1)$-valued random variables
 with distribution function $(0,1)\ni b\mapsto b^\alpha$.
 Define for every $n=1,\ldots,N$, $\theta\in \Theta^N_{N-n}$
 \begin{equation}\label{eq:def_MLP_alg}
 \begin{aligned}
  &\begin{pmatrix} \widetilde{v}^\theta_{n,M}(t,x) \\ \widetilde{v}^\theta_{n,M,\nabla}(t,x) \end{pmatrix}
  \\
  &\quad:=\begin{pmatrix}\phi_{\Psi,\delta_\Psi}(x) \\ 0 \end{pmatrix}
  + 
  \frac{1}{M^n}\sum_{i=1}^{M^n}
  \left(\phi_{\Psi,\delta_\Psi}\left(x + \sqrt{t_f-t} Z^{(\theta,0,-i)} \right)     - \phi_{\Psi,\delta_\Psi}(x)\right) 
  \begin{pmatrix} 1 \\ (t_f-t)^{-1/2} Z^{(\theta,0,-i))} \end{pmatrix} 
  \\
  &\qquad+
  \sum_{l=0}^{n-1} \sum_{i=1}^{M^{n-l}}
  \frac{(t_f-t)(\tau^{(\theta,l,i)})^{1-\alpha} }{\alpha M^{n-l}}
  \begin{pmatrix} 1 \\ [(t_f - t )\tau^{(\theta,l,i)}]^{-1/2} Z^{(\theta,l,i)} \end{pmatrix}
  \\
  &\qquad \times \left[\phi_{H,\delta_H,R}\left(t', x', v^{(\theta,l,i)}_{l,M,\nabla}(t',x')   \right)
  - \mathbbm{1}_\bbN(l) \phi_{H,\delta_H,R}\left(t', x', v^{(\theta,-l,i)}_{l-1,M,\nabla}(t',x')   \right)\right]
  \Bigg\vert_{(t',x') = (t^{(\theta,l,i)},x^{(\theta,l,i)}_t)}
  ,
  \end{aligned}
 \end{equation}
where for every $\theta \in \Theta^N_{N}$ and every $(t,x)\in [0,t_f]\times \bbR^d$
\begin{equation*}
(\widetilde{v}^\theta_{0,M}(t,x), \widetilde{v}^\theta_{0,M,\nabla}(t,x))^\top 
  = (\widetilde{v}^\theta_{-1,M}(t,x), \widetilde{v}^\theta_{-1,M,\nabla}(t,x))^\top=0
\end{equation*}  
  and for every $\theta \in \Theta^N_{N-n}$, $n=1,\ldots,N$, and every $(t,x)\in [0,t_f]\times \bbR^d$
  \begin{equation*}
      x^\theta_t =
      x + \sqrt{(t_f-t)\tau^\theta}Z^\theta
      \quad\text{and}\quad 
      t^\theta = 
      t + (t_f - t )\tau^\theta .
  \end{equation*}
Let $Q\subset\bbR^d$ be a bounded domain.
For every $\varepsilon\in (0,1)$ there exists $N_\varepsilon\in\bbN$ such that for every $N\geq N_\varepsilon$ 
\begin{equation*}
\sup_{\delta_H,\delta_{\Psi}\in (0,1)}
 \sqrt{\bbE\left( \left\|\widetilde{V}(0,\cdot) - \widetilde{v}^0_{N,\lfloor N^\alpha\rfloor}(0,\cdot)   \right\|^2_{L^2(Q)}\right) }
 +
 \sqrt{\bbE\left( \left\|\|\nabla_x \widetilde{V}(0,\cdot) - \widetilde{v}^0_{N,\lfloor N^\alpha\rfloor,\nabla}(0,\cdot) \|_2   \right\|^2_{L^2(Q)}\right) }
 \leq 
 \varepsilon
\end{equation*}
and for every $\delta>0$,
$\sup_{\delta_H,\delta_{\Psi}\in (0,1)}\# (\Theta^{N_\varepsilon}_{N_\varepsilon}) = \calO(|Q|^{1+\delta/2}d^\kappa\bar{d}^\kappa \varepsilon^{-2+\delta}) $
for some constant $\kappa$ that neither depends on $d$ nor on $\bar{d}$.
 
\end{proposition}

\begin{proof}
The error estimate follows for every $x\in Q$ pointwise 
by~\cite[Equation~(132)]{HJK_2019}. 
Then, for every $\varepsilon'\in (0,1)$
there exist $N_{\varepsilon'}(x)$, $x\in Q$,
such that for every $N\geq \sup_{x\in Q}N_{\varepsilon'}(x) =: \bar{N}_{\varepsilon'}$,
\begin{equation*}
 \sqrt{\bbE\left( \left\|\widetilde{V}(0,\cdot) - \widetilde{v}^0_N(0,\cdot)   \right\|^2_{L^2(Q)}\right) }
 +
 \sqrt{\bbE\left( \left\|\|\nabla_x \widetilde{V}(0,\cdot) - \widetilde{v}^0_{N,\nabla}(0,\cdot) \|_2  \right\|^2_{L^2(Q)}\right) }
 \leq 
 \varepsilon \sqrt{|Q|}
\end{equation*}
By~\cite[Equation~(133)]{HJK_2019},
\begin{equation*}
\begin{aligned}
  \# (\Theta^{\bar{N}_{\varepsilon'}}_{\bar{N}_{\varepsilon'}})
  &\leq 
  C d 
  (\varepsilon')^{-(2+\delta)}
  \left [
  1+ \sqrt{t_f} C_{\phi_\Psi,x}d + \sup_{x\in \bbR^d}|\phi_\Psi(x)|
  + C_{\phi_\Psi,x} + t_f C_{H,x} 
  \right.
  \\
  &\quad  \left.
  + t_f \sup_{t\in [0,t_f],x\in\bbR^d}|\phi_{H,\delta_H,R}(t,x,0)| 
  + t_f d C_{\phi_H,p}(C_{\phi_\Psi,x} + t_f C_{\phi_H,x})
  \right]^{2+\delta}
  ,
  \end{aligned}
\end{equation*}
where $C>0$ is a constant that only depends on $\alpha$
and $C_{\phi_H,p}$, $C_{\phi_H,x}$, and 
$C_{\phi_\Psi,x}$
are the corresponding Lipschitz constants 
of $\phi_{H,\delta_H,R}$
and $\phi_{\Psi,\delta_\Psi}$
with respect to the $1$-norm.
These Lipschitz constants are independent from 
$\delta_H,\delta_\Psi\in (0,1)$.
By Proposition~\ref{prop:DNN_Lipschitz} and~\eqref{eq:bound_R},
there exist $C',\kappa' >0$
such that for every $d,\bar{d}\in \bbN$,
\begin{equation*}
    C_{\phi_H,x} + C_{\phi_H,p}
    \leq 
    C'
    d^{\kappa'}
    \bar{d}^{\kappa'}
    .
\end{equation*}
Lemma~\ref{cor:bound_phi_H} and the assumptions of this proposition
imply that
there exist $C',\kappa' >0$
such that for every $d,\bar{d}\in \bbN$,
\begin{equation*}
    \sup_{\delta_H\in(0,1)}
    \sup_{t\in [0,t_f], x',p' \in \bbR^d}
    |\phi_{H,\delta_H,R}(t,x',p')|
    +|\phi_{\Psi,\delta_\Psi}(x')|
    \leq
    C'
    d^{\kappa'}
    \bar{d}^{\kappa'}
    .
\end{equation*}
The assertion follows upon the choice $\varepsilon' = \varepsilon/ \sqrt{|Q|}$.
\end{proof}

\begin{theorem}\label{thm:DNN_approx_value_fct}
Let the assumptions of Proposition~\ref{prop:MLP_approx}
be satisfied.
Suppose that for every $\delta_\Psi \in (0,1)$ the
ReCU DNNs $\phi_{\Psi,\delta_\Psi}$
satisfies
that for every $x\in\bbR^d$
\begin{equation*}
    |\Psi(x) - \phi_{\Psi,\delta_\Psi}(x)|
    \leq \delta_\Psi(1 + \|x\|_2^q)
\end{equation*}
and ${\rm size}(\phi_{\Psi,\delta_\Psi}) = \calO((d\bar{d})^{\bar{\kappa}} \delta_\Psi^{-\bar{\kappa}})$ for some fixed $\bar{\kappa},q\geq 1$.
Then, for every $\varepsilon>0$ there exist DNNs $\phi_\varepsilon$ and $\phi_{\varepsilon,\nabla}$
such that
\begin{equation*}
 \| V(0,\cdot) - \phi_\varepsilon\|_{L^2(Q)} 
 +
 \left\|\|\nabla_x V(0,\cdot) - \phi_{\varepsilon,\nabla}\|_2\right\|_{L^2(Q)}
 \leq \varepsilon
\end{equation*}
and ${\rm size}(\phi_\varepsilon) = \calO(\sup_{x\in Q}\{\|x\|_2^\kappa\} |Q|^\kappa d^\kappa \bar{d}^\kappa \varepsilon^{-\kappa})$ and 
${\rm size}(\phi_{\varepsilon,\nabla}) = \calO(\sup_{x\in Q}\{\|x\|_2^\kappa\} |Q|^\kappa d^\kappa \bar{d}^\kappa \varepsilon^{-\kappa})$ 
for some $\kappa>0$ that does not depend on $d,\bar{d}$.
\end{theorem}

\begin{proof}
 This proof will be carried out in two steps. 
 First, we study the perturbation by the approximation of the Hamiltonian and the terminal condition.
 Then, we verify the consistent approximation of the value function by  DNNs without incurring the curse of dimension.
 
 The error between $V$ and $\widetilde{V}$
 and between the gradients can be controlled by Lemma~\ref{lem:perturbation_value_function}.
 In order to be able to apply Lemma~\ref{lem:perturbation_value_function},
 we need to verify the assumptions in \eqref{eq:approx_H_Psi}, \eqref{eq:Lipschitz_H_Hstar},
 \eqref{eq:Lipschitz_H_R*}, and \eqref{eq:Lipschitz_Psi_R*}.
 The assumptions in \eqref{eq:approx_H_Psi} and~\eqref{eq:Lipschitz_H_Hstar} with $\Psi^* = \phi_{\Psi,\delta_\Psi}$
 and $H_{R}^*= \phi_{H,\delta_H,R}$ are satisfied according to Proposition~\ref{prop:DNN_approx_H}, Remark~\ref{rmk:DNN_phi_H_R}, and Corollary~\ref{cor:bound_phi_H}.
 The assumption in~\eqref{eq:Lipschitz_H_R*} is implied by Proposition~\ref{prop:DNN_Lipschitz}
 and the assumption in~\eqref{eq:Lipschitz_Psi_R*} 
 is assumed here in~\eqref{eq:Lipschitz_DNN_Psi}.
 Furthermore, as a consequence of Assumption~\ref{ass:lineardynamics}(i),(iv) and~\eqref{eq:bound_R}, 
 there exist $C,\kappa>0$ such that
 for every $d,\bar{d}\in\bbN$,
 \begin{equation*}
 \sup_{t\in [0,t_f]}\sup_{x',p'\in\bbR^d}|H_R(t,x',p')|
 \leq 
 C
 d^\kappa \bar{d}^\kappa.
 \end{equation*}
 Thus, there exists $\kappa>0$ such that for every $\delta\in (0,1)$ 
 \begin{equation*}
 \| V(0,\cdot) - \widetilde{V}(0,\cdot)\|_{L^2(Q)} 
 +
 \left\|\|\nabla_x V(0,\cdot) - \nabla_x\widetilde{V}(0,\cdot)\|_2\right\|_{L^2(Q)}
 \leq \delta
 \end{equation*}
 with $\max\{{\rm size}(\phi_{H,\delta_H,R}),{\rm size}(\phi_{\Psi,\delta_\Psi})  \} = \calO(|Q|^\kappa\sup_{x\in Q}\{\|x\|_2^\kappa\}d^{\kappa} \bar{d}^{\kappa}  \delta^{-\kappa})$.
 Note that for every $x\in\bbR^d$, $\|x\|_2 \leq \sqrt{d}\|x\|_\infty$.

 Proposition~\ref{prop:MLP_approx} states an error estimate in a root mean squared sense.
 Recall the fact that for any positive random variable $X$ that satisfies $\bbE(X)\leq c$ for some $c>0$
 there exists a measurable set $\Omega'\subset\Omega$ such that 
 $X(\omega)\leq c$ for every $\omega\in \Omega'$ such that
 $\bbP(\Omega')>0$.
 Let $\delta'\in (0,1)$ be arbitrary and to be determined below.
 As a result there exists $\Omega'$ with $\bbP(\Omega')>0$ and
 \begin{equation*}
  \|\widetilde{V}(0,\cdot) - \widetilde{v}^0_{N,\lfloor N^\alpha\rfloor}(0,\cdot)(\omega)   \|_{L^2(Q)} 
 +
 \left\|\|\nabla_x \widetilde{V}(0,\cdot) - \widetilde{v}^0_{N,\lfloor N^\alpha\rfloor,\nabla}(0,\cdot) (\omega) \|_2  \right\|_{L^2(Q)} 
 \leq 
 \delta'
\end{equation*}
for every $\omega\in \Omega'$.
Since the definition of the mappings $\widetilde{v}^0_{N,\lfloor N^\alpha\rfloor}$ 
and $\widetilde{v}^0_{N,\lfloor N^\alpha\rfloor,\nabla} $
in Proposition~\ref{prop:MLP_approx} requires only composition and addition of realizations 
of DNNs, there exist DNNs
$\phi_\varepsilon$ and $\phi_{\nabla,\varepsilon}$ 
such that for every $x\in Q$ and fixed $\omega\in \Omega'$
\begin{equation*}
    \phi_\varepsilon(x) = \widetilde{v}^0_{N,\lfloor N^\alpha\rfloor}(0,x)(\omega)
    \quad \text{and} \quad 
    \phi_{\nabla,\varepsilon}(x) = \widetilde{v}^0_{N,\lfloor N^\alpha\rfloor,\nabla}(0,x) (\omega)
    .
\end{equation*}
Thus, the assertion of the error estimate follows upon 
choosing $\delta\simeq \varepsilon\simeq \delta'$.
The assertion on the size of the DNNs
$\phi_\varepsilon$ and $\phi_{\nabla,\varepsilon}$ follows by the estimate on the size of 
$\phi_{H,\delta,R}$ in Proposition~\ref{prop:DNN_approx_H} and Remark~\ref{rmk:DNN_phi_H_R}, by the assumption on the size of $\phi_{\Psi,\delta}$, 
and by Proposition~\ref{prop:MLP_approx}, which contains an upper bound on the number of scalar-valued random variables
that are used to sample $\widetilde{v}^0_{N,\lfloor N^\alpha\rfloor}(0,x)$ 
and $\widetilde{v}^0_{N,\lfloor N^\alpha\rfloor,\nabla}(0,x)$.
\end{proof}

\begin{corollary}\label{cor:DNN_approx_control}
Let the assumptions of Theorem~\ref{thm:DNN_approx_value_fct}
be satisfied.
For every $\varepsilon\in (0,1)$, 
there exists a DNN $\phi_{\varepsilon,\alpha}$
such that
\begin{equation*}
     \left\|\|\alpha(0,\cdot) - \phi_{\varepsilon,\alpha}\|_2\right\|_{L^2(Q)}
 \leq \varepsilon
\end{equation*}
and ${\rm size}(\phi_{\varepsilon,\alpha}) = \calO(\sup_{x\in Q}\{\|x\|_2^\kappa\} |Q|^\kappa d^\kappa \bar{d}^\kappa \varepsilon^{-\kappa})$ 
for some $\kappa>0$ 
that does not depend on $d,\bar{d}$.
\end{corollary}

\begin{proof}
 It follows similarly to Lemma~\ref{lem:formula_H} that
\begin{equation*}
    \alpha(0,x)_i
    =
    \min\{ \max\{- (f_2(0,x)^\top \nabla_x V(0,x) )_i /(2\gamma)  , a_i\}, b_i\} 
    \quad 
    i=1,\ldots, \bar{d}
    .
\end{equation*}
Consider the DNNs
\begin{equation*}
    \phi_{\varepsilon,\alpha}(x)_i  
    = 
     \min\{ \max\{- (\tilde{\times}_{\varepsilon_0}(\phi_{f_2,\delta_2}(0,x)^\top, \phi_{\delta,\nabla}(0,x)) )_i /(2\gamma)  , a_i\}, b_i\} 
    \quad 
    i=1,\ldots, \bar{d}.
\end{equation*}
By~\eqref{eq:Lipschitz_min_max}, 
\begin{equation*}
\begin{aligned}
   \left\|\|\alpha(0,\cdot) -  \phi_{\varepsilon,\alpha}\|_2\right\|_{L^2(Q)}
    &\leq 
    \frac{\sqrt{\bar{d}}} {2\gamma}
    \max_{i=1,\ldots,\bar{d}}  
    \min\left\{\max\{\|a\|_\infty, \|b\|_\infty\},   \ldots 
    \vphantom{ \|(f_2(0,x)^\top \nabla_x V(0,x) )_i
    -
    (\tilde{\times}_{\varepsilon_0}(\phi_{f_2,\delta_2}(0,x)^\top, \phi_{\delta,\nabla}(0,x)) )_i\|_{L^2(Q)}}
    \right.
    \\
    &\qquad\left.
    \|(f_2(0,x)^\top \nabla_x V(0,x) )_i
    -
    (\tilde{\times}_{\varepsilon_0}(\phi_{f_2,\delta_2}(0,x)^\top, \phi_{\delta,\nabla}(0,x)) )_i\|_{L^2(Q)}
    \right\}.
    \end{aligned}
\end{equation*}
The assertion now follows by Lemma~\ref{lem:ReLU_Matr_vect_prod}, Theorem~\ref{thm:DNN_approx_value_fct}, 
and the assumed approximation properties of the DNNs $\phi_{f_2,\delta_2}$.
\end{proof}

Theorem~\ref{thm:main_resutl} is now implied by Theorem~\ref{thm:DNN_approx_value_fct}
and Corollary~\ref{cor:DNN_approx_control}, which concludes the theoretical part of this manuscript.

\section{Conclusions}

This paper provides the first DNN approximation results for stochastic optimal control problems that are free of the curse of dimension. In particular our results confirm that current numerical approaches for the solution of high dimensional HJB equations constitute a viable approach. Natural next questions include the generalization of our results to more general stochastic optimal control problems as well as the investigation of possible algorithmic realizations of the convergence rates that our approximation results provide. We leave these for future work.

\begin{appendix}
\section{A priori Estimates}
\label{sec:apriori_est}

We provide quantitative bounds on the value function and the gradient
in the following lemma.
Note that the Hamiltonian $H$ is not globally Lipschitz continuous here. 
So, the given estimates are non-trivial.
\begin{lemma}[Theorem~A.1(iii) of \cite{HJB_MComp_2007}]
\label{lem:bound_grad_V}
Suppose that there exist constants $C^d_1,C^{d,\bar{d}}_2$, $d,\bar{d}\in\bbN$, such that for every  $d,\bar{d}\in\bbN$
and every $x,x'\in\bbR^d$
\begin{equation*}
    \sup_{s\in [0,t_f]}\|f_1(s,x)
    -f_1(s,x')\|_2
    \leq C^d_1\|x-x'\|_2
\end{equation*}
and
\begin{equation*}
    \sup_{s\in [0,t_f]}\|f_2(s,x)
    -f_2(s,x')\|_2
    \leq C^{d,\bar{d}}_2\|x-x'\|_2
    .
\end{equation*}
It holds that
\begin{equation*}
    \sup_{s\in [0,t_f]}\sup_{x\in \bbR^d}|V(s,x)|\leq 
    e^{c_0 t_f }
    \left(
    \sup_{x\in\bbR^d}\|\Psi(x)\|_2
    +t_f \sup_{s\in[0,t_f]} \sup_{x\in\bbR^d}
    \bar{L}(x,s) + \gamma\max\{\|a\|^2_2,\|b\|^2_2\}\right),
\end{equation*}
where $c_0= \sup_{s\in [0,t_f],x\in\bbR^d}\|f^d_1(s,x)\|_2 $,
and
\begin{equation*}
\begin{aligned}
    &\sup_{s\in [0,t_f]}\sup_{x\in \bbR^d}\|\nabla_x V(s,x)\|_2
    \\
    &\leq 
    e^{c_1 t_f}
    \left( 
    \sup_{x\in\bbR^d}\|\nabla_x \psi(x)\|_2
    + t_f \sup_{s\in [0,t_f]}\sup_{x\in \bbR^d}|V(s,x)|
    (\sup_{s\in [0,t_f]}\sup_{x\in\bbR^d}
    \|f_2(s,x)\|_2 + C^{d,\bar{d}}_2)
    \right.
    \\
    &\quad
    \left.
    \vphantom{\sup_{x\in\bbR^d}}
    +
    t_f
    \sup_{s\in [0,t_f]}\sup_{x\in \bbR^d}(
    |\bar{L}(s,x)|+
    \|\nabla_x \bar{L}(s,x)\|_2)+ \gamma\max\{\|a\|_2^2,\|b\|_2^2\}
    \max\{\|a\|_\infty,\|b\|_\infty\} 
    \right)
    ,
\end{aligned}
\end{equation*}
where 
$c_1 = 
1/4 
+ \sup_{s\in [0,t_f]}\sup_{x\in\bbR^d}
[\|f_1(s,x)\|_2 + \|f_2(s,x)\|_2
\max\{\|a\|_\infty,\|b\|_\infty\}]
+
C^d_1$.
\end{lemma}
We state an immediate corollary of this lemma. 
\begin{corollary}\label{cor:bound_grad_V}
Let the assumptions of Lemma~\ref{lem:bound_grad_V}
be satisfied.
Suppose that there exist $C>0$ and $\kappa_1,\kappa_2>0$
such that for every $d,\bar{d}\in \bbN$,
\begin{equation*}
    \sup_{s\in [0,t_f]}\sup_{x\in\bbR^d}
    \|f_1(s,x)\|_2 
    + 
    \|f_2(s,x)\|_2 
    +C_1^d
     \leq 
     C 
\end{equation*}
and 
\begin{equation*}
    \sup_{x\in\bbR^d}|\Psi(x)|
    +
    \sup_{x\in\bbR^d}\|\nabla_x\Psi(x)\|_2
    +
    \sup_{s\in [0,t_f]}\sup_{x\in\bbR^d}
    (\bar{L}(s,x)
    +
    \|\nabla_x \bar{L}(s,x)\|_2)
    +
    C_2^{d,\bar{d}}
    \leq 
    C d^{\kappa_1} \bar{d}^{\kappa_2} 
    .
\end{equation*}
Then, there exist $\bar{\kappa},\kappa_3,\kappa_4>0$ and $C'>0 $
such that for every $d,\bar{d}\in\bbN$ 
\begin{equation*}
    \sup_{s\in [0,t_f]}\sup_{x\in\bbR^d}
    \|\nabla_x V(s,x)\|_2
    \leq 
    C' e^{\bar{\kappa} t_f}
    d^{\kappa_3} \bar{d}^{\kappa_4} 
    .
\end{equation*}
\end{corollary}

\section{Classical solutions to certain quasi-linear parabolic equations}

We consider the following quasi-linear parabolic equation
\begin{equation}\label{eq:appendix_semilinear_PDE}
    \partial_t W(t,x) + \frac{1}{2}\Delta W(t,x) 
    + G(t,x,\nabla_x W(t,x)) = 0 \text{ on } [0,t_f)\times \bbR^d,
    \quad
    W(t_f,x) = \Phi(x) \text{ on } \bbR^d .
\end{equation}
We assume that $G$ is globally Lipschitz continuous and suppose that 
$\Phi \in C^2(\bbR^d)$ is at most linearly growing.

\begin{proposition}\label{prop:appendix_classical_sol}
 There is a unique solution 
 $W$ to~\eqref{eq:appendix_semilinear_PDE} and 
 $W\in C^{1,2}([0,t_f]\times \bbR^d)$.
\end{proposition}
\begin{proof}
We begin this proof by showing that there exists a unique viscosity solution $W$ to 
\eqref{eq:appendix_semilinear_PDE} which is Lipschitz continuous with respect to $x$. Then, by a bootstrap argument we establish the claimed regularity of $W$.

Since $G$ is globally Lipschitz continuous and $\Phi$ is at most linearly growing,
\cite[Theorem~5.1]{PardouxTang_1999} implies that there exists a unique viscosity solution $W$ 
to~\eqref{eq:appendix_semilinear_PDE}.
The second estimate in~\cite[Equation~(4.22)]{PardouxTang_1999}
implies that $W$ is Lipschitz continuous with respect to $x$; here $W(t,x) = Y^{t,x}(t)$ in the notation of~\cite{PardouxTang_1999}.

 The Lipschitz continuity of $W(t,\cdot)$ 
 implies that for every $t\in [0,t_f]$, 
 $\nabla_x W(t , \cdot)$ has essentially bounded entries.
 Let $x\in \bbR^d$ and $s\in [0,t_f]$ be arbitrary and $r>0$. 
 Let $\chi$ be a smooth function such that 
 $\chi:B_{4r}(x) \to \bbR$,
 $\chi\vert_{B_{r}(x)} =1$ and $\chi $ is compactly supported in $B_{2r}(x)$. 
 Note that 
 for any $x'\in\bbR^d$ and any $r'>0$ we denote 
 the Euclidean ball with center $x'$ and radius $r'$
 by $B_{r'}(x'):=\{x''\in\bbR^d : \|x'' - x'\|_2 < r'\}$.
 Consider the linear parabolic PDE
 on $[0,t_f)\times B_{4r}(x)$
 \begin{equation}\label{eq:appendix_local_PDE}
     \partial_t \tilde{w}  + \frac{1}{2}\Delta \tilde{w} = F, 
     \quad \tilde{w}\vert_{[0,t_f) \times \partial B_{4r}(x)} = 0, 
     \quad \tilde{w}(t_f,\cdot) = \chi\Phi, 
 \end{equation}
 where $F(t,x') = G(t,x,\nabla_x (\chi(x') W(t,x')))$ for every $t\in [0,t_f]$, $x'\in \bbR^d$.
 Note that $F\in L^\infty([0,t_f] \times B_{4r}(x) )$.
 It is apparent that $\chi W\vert_{[0,t_f] \times B_{4r}(x)}$ is a solution 
 to~\eqref{eq:appendix_local_PDE}
 and $\chi W\vert_{[0,t_f] \times B_{r}(x)} = W\vert_{[0,t_f] \times B_{r}(x)}$.
 
 According to~\cite[Theorem~6.6]{Lieberman}
 and~\cite[Theorem~7.14]{Lieberman}, it holds that $\tilde{w}, \partial_{x_j} \tilde{w}  \in W^{1,p}(\mathcal{D})$
 for any compactly included $\mathcal{D}\subset\subset [0,t_f]\times B_{4r}(x)$ 
 and every $p\in [2,\infty)$ and $j=1,\ldots,d$ such that $\{(s,x)\}\subset\subset \mathcal{D}$.
 Let $s_1,s_2$ satisfy $0<s_1<s<s_2<t_f$.
 Choose $\mathcal{D} = [s_1, s_2)\times B_{3r}(x)$.
 By the Sobolev embedding theorem 
 \cite[Theorem 1.107]{Triebel_Fct_Sp_III},
 $\tilde{w}, \partial_{x_j} \tilde{w}  \in C^{\alpha}(\overline{\mathcal{D}})$
 for any $\alpha\in [0,1)$
 and $j=1,\ldots, d$.
 By the continuity of the value function $W$ and uniqueness of the solution to~\eqref{eq:appendix_local_PDE},
 it holds that $\chi W\vert_{\mathcal{D}} = \tilde{w}\vert_{\mathcal{D}}$ and therefore $F\vert_{\overline{\mathcal{D}}} \in C^{\alpha}(\overline{\mathcal{D}})$.
 Now the regularity theory on interior H\"older estimates for parabolic PDEs becomes applicable.
 In particular, we consider the PDE in~\eqref{eq:appendix_local_PDE} 
 on the space-time domain $\mathcal{D}$ and also recall that $\{(s,x)\}\subset\subset \mathcal{D}$.
 Let us construct a sequence of solutions by mollifying the right hand side $F \in C^{\alpha}(\overline{\mathcal{D}})$
 to obtain $F_k$, $k\in \bbN$, such that $F_k \to F $ in the $C^{\bar{\alpha}}(\overline{\mathcal{D}})$-norm for any $\bar{\alpha}\in (0,\alpha)$ and $F_k$ is smooth
 for every $k\in \bbN$.
 Let us denote the associated solutions by $\tilde{w}_k$ which solve the parabolic PDE 
 on $[s_1,s_2)\times B_{3r}(x)$
  \begin{equation*}
     \partial_t \tilde{w}_k  + \frac{1}{2}\Delta \tilde{w}_k = F_k, 
     \quad \tilde{w}_k\vert_{[s_1,s_2) \times \partial B_{3r}(x)} = 0, 
     \quad \tilde{w}(s_2,\cdot) = g_k(s_2,\cdot), 
 \end{equation*}
 where $g_k$, $k\in \bbN$, is smooth and converges to $\chi W\vert_{[s_1,s_2]\times\overline{B_{3r}(x)}}$
 in the $C^0([s_1,s_2]\times\overline{B_{3r}(x)})$-norm.
 By a classical result, e.g.~\cite[Theorem~7.7]{EvansPDEs}, the solutions $\tilde{w}_k$ are smooth for every $k\in \bbN$.
 The regularity estimate in \cite[Theorem~4.9]{Lieberman} implies that
 $w_k$ is a Cauchy sequence in a space-time H\"older space.
 Particularly, for any subdomain $\mathcal{D}'$ such that 
 $\{(s,x)\} \subset\subset \mathcal{D}' \subset \subset [s_1, s_2)\times B_{3r}(x)$
 there exists a constant $C>0$ such that for every $k,k'\in \bbN$, 
 \begin{equation*}
 \begin{aligned}
     &\|\tilde{w}_k - \tilde{w}_{k'}\|_{C^{1+\alpha/2}(\overline{\mathcal{D}'})}
     +
     \max_{i=1,\ldots,d}\{\|\partial_{x_i}\tilde{w}_k \|_{C^{1+\bar{\alpha}/2}(\overline{\mathcal{D}'})}\}
     \\
     &\qquad 
     \leq 
     C
     (\|\tilde{w}_k - \tilde{w}_{k'}\|_{C^0([s_1,s_2]\times\overline{B_{3r}(x)}))} + \|F_k - F_{k'}\|_{C^{\bar{\alpha}}([s_1,s_2]\times\overline{B_{3r}(x)})})
 .
 \end{aligned}
 \end{equation*}
 The term $\|\tilde{w}_k - \tilde{w}_{k'}\|_{C^0([s_1,s_2]\times\overline{B_{3r}(x)}))}$ converges to zero as $k,k' \to \infty$
 by~\cite[Theorem~2.10]{Lieberman}, 
 where we used the linearity of the PDE and the fact that 
 for every $k,k'\in\bbN$,
 \begin{equation*}
 \|g_k - g_{k'}\|_{C^0([s_1,s_2]\times \partial B_{3r}(x)\cup \{s_2\}\times B_{3r}(x) ) }
 \leq
 \|g_k - g_{k'}\|_{C^0([s_1,s_2]\times\overline{B_{3r}(x)})}
 .
 \end{equation*}
 Hence, the limit of $\tilde{w}_k$ exists and is in $C^{1,2}(\overline{\mathcal{D}'})$.
 Choose $\mathcal{D}' := [\bar{s}_1,\bar{s}_2) \times B_{2r}(x)$,
 where $\bar{s}_1,\bar{s}_2$ satisfy $s_1<\bar{s}_1<s<\bar{s}_2<s_2$.
 Since the convergence is also uniform, the limit $\bar{w}:=\lim_{k\to\infty} \tilde{w}_k$
 solves 
 \begin{equation*}
     \partial_t \bar{w}  + \frac{1}{2}\Delta \bar{w} = F, 
     \quad \bar{w}\vert_{[\bar{s}_1,\bar{s}_2) \times \partial B_{2r}(x)} = 0, 
     \quad \bar{w}(\bar{s}_2,\cdot) = \chi W(\bar{s}_2,\cdot)\vert_{B_{2r}(x)}, 
 \end{equation*}
 which is also solved by $\chi W\vert_{[s_1,s_2]\times \overline{B_{2r}(x)}}$. 
 By uniqueness it follows that $\bar{w}=\chi W\vert_{[s_1,s_2]\times \overline{B_{2r}(x)}}$.
 Since $\chi\vert_{B_{r}(x)}=1$ and $\bar{s}_1<s<\bar{s}_2$, it follows that 
 $W$ is of class $C^{1,2}$ at the point $(s,x)$. 
 Since $(s,x)\in [0,t_f)\times \bbR^d$ 
 was arbitrary and since $W(t_f,\cdot)=\Phi \in C^2(\bbR^d)$ by assumption, 
 the assertion of this proposition is proved.
\end{proof}

\end{appendix}


\end{document}